\documentclass[10pt,a4paper]{amsart}
\usepackage[utf8]{inputenc}
\usepackage{amsmath}
\usepackage{amsfonts}
\usepackage{amssymb}
\usepackage{amsthm}
\usepackage{tikz-cd}
\usepackage{hyperref}
\usepackage{cleveref}
\usepackage{url}
\newtheorem{thm}{Theorem}[section]
\newtheorem{prop}[thm]{Proposition}
\newtheorem{lem}[thm]{Lemma}
\newtheorem{cor}[thm]{Corollary}
\newtheorem{defn}[thm]{Definition}
\theoremstyle{definition}
\newtheorem{rem}[thm]{Remark}
\theoremstyle{definition}
\newtheorem{ex}[thm]{Example}
\newcommand{\catname}[1]{\textup{\textbf{#1}}}
\newcommand{\infcatname}[2]{\mathcal{#1}\textup{#2}}
\newcommand{\Set}{\catname{Set}}

\newcommand{\Grp}{\catname{Grp}}
\newcommand{\sSet}{\catname{sSet}}
\newcommand{\sCat}{\catname{sCat}}

\newcommand{\biginfcat}{\widehat{\infcatname{C}{at}}_{\infty}}

\newcommand{\Fun}{\textup{Fun}}

\newcommand{\id}{\textup{id}}

\newcommand{\N}{\textup{N}}

\newcommand{\Ord}{\textup{\textbf{Ord}}}
\newcommand{\vp}{\textup{VP}}
\newcommand{\wvp}{\textup{WVP}}
\newcommand{\swvp}{\textup{SWVP}}
\newcommand{\ivp}{\textup{VP}_{\infty} }
\newcommand{\iwvp}{\textup{WVP}_{\infty}}
\newcommand{\iswvp}{\textup{SWVP}_{\infty}}
\title{Vop\v{e}nka's principle in $\infty$-categories}
\author{Giulio Lo Monaco}
\thanks{The author was supported by the Grant Agency of the Czech Republic under the grant 19-00902S}
\date{}

\begin{document}

\begin{abstract}
In this article, the interplay between Vop\v{e}nka's principle, as well as its weaker counterpart, and presentable $\infty$-categories is studied. Analogous statements, arising after replacing categories with $\infty$-categories in the original ones, are introduced and compared to these. Further, the attention is focused on the question of to what extent the consequences that (weak) Vop\v{e}nka's principle have on the detection of reflective subcategories of a presentable categories can be generalized to $\infty$-categories.
\end{abstract}

\maketitle
\tableofcontents

\section*{Introduction}
The purpose of this article is to investigate the behavior of $\infty$-categories in relation to (weak) Vop\v{e}nka's principle. The oldest version of it was introduced somewhat skeptically by Petr Vop\v{e}nka, who formulated it in the language of logic and at first thought of it as a sort of mockery, as he didn't believe it to be consistent with ZFC set theory, although he himself found a flaw in his proof of its inconsistency. The statement was then written up in \cite{SolovayReinhardtKanamori} and made popular in the first edition of \cite{Jech}, and a few other equivalent formulations were proven. Much of this is expounded in a set-theoretical language in \cite{Kanamori}.\\
Outside of set theory, there are a number of other equivalent formulations of Vop\v{e}nka's principle that instead take place in category theory, more in particular in the context of presentable categories. To my knowledge, the most complete survey of categorical facts that depend on, or relate to Vop\v{e}nka's principle in this other flavor is Chapter 6 of \cite{LocPresAccCat}. Moreover, as it is explained in the appendix of \emph{loc. cit.}, Vop\v{e}nka's principle is independent of ZFC set theory, and it actually occupies a very high rank in the hierarchy of large cardinals.
Section 6.D of \cite{LocPresAccCat} explores some of the consequences of Vop\v{e}nka's principle concerning the existence, or the description, of reflective subcategories of presentable ones, and it is to this section that most of our attention will be directed in the present study, whose main purpose is reformulating Vop\v{e}nka's principle, in one of its equivalent forms having the most plainly categorical spirit, in the context of $\infty$-categories, and asking ourselves the question whether analogous results to those detailed in \emph{loc. cit.} can be obtained in the homotopical case.\\
The statement known as weak Vop\v{e}nka's principle stems from one of the equivalent categorical formulations of its stronger counterpart, that is spelled out in \cite{LocPresAccCat}, Lemma 6.3(iii). It turns out that, by a quirk of mathematical nature, if the stronger Vop\v{e}nka's principle can be expressed by forbidding the existence of full embeddings of $\Ord$ into presentable categories, using $\Ord^{op}$ instead yields a statement which is still very high in the aforementioned hierarchy (\cite{LocPresAccCat}, Appendix), but strictly weaker than the original statement, as was shown in \cite{VP-WVP}. We will also analyze the behavior of this weaker principle and some of its consequences when transferred to $\infty$-categories.\\
By another quirk of mathematical nature, it so happens that weak Vop\v{e}nka's principle may be reformulated in an apparently loosened version, known as semi-weak Vop\v{e}nka's principle, which has recently been proven to be equivalent to the weak one (\cite{VP-WVP}). Among other consequences, this provides a relatively simple way to show that weak Vop\v{e}nka's principle can be equivalently formulated in 1-categories or $\infty$-categories, whereas a similar result is until now not known for Vop\v{e}nka's principle, which leaves open the question whether its 1-categorical and $\infty$-categorical formulations are equivalent.

\section*{Note on size issues}
As often in category theory, the set-theoretical issue of the size of the objects one is dealing with needs to be addressed one way or another, and expecially so in the present work, where the main statements that are taken in consideration are based on a distinction between certain collections of objects being “small” or “large”, so I believe a further clarification is in order.\\
One possible way is adopting NBG set theory, where the basic notion is that of class. A class is called a set if it can be an element of another class, and in particular all sets are subject to the axioms of ZFC set theory. On the other hand, we refer to a class as proper if it cannot be an element of another class. In this setting all categories are assumed to be comprised of a class of objects and a class of morphisms with the usual structure, and an $\infty$-category should be a functor $\Delta^{op} \to \catname{Class}$ with the usual inner horn lifting property; they are called small if all the classes in their definitions are in fact sets. Note that, if we want to be pedantic, these two definitions are not permitted unless we have an additional tier of collections that are even larger than classes. These are sometimes called conglomerates and in particular a conglomerate is said to be a class if it can be an element of another conglomerate.\\
Another approach is to adopt TG set theory and fixing in particular two Grothendieck universes $\mathcal{U} \in \mathcal{V}$. Elements of $\mathcal{U}$ will be called small, elements of $\mathcal{V}$ will be called large, a category will be comprised of two $\mathcal{V}$-sets plus the usual structure and an $\infty$-category is a functor $\Delta^{op} \to \mathcal{V}$-$\Set$ with the inner horn lifting property. They are called small if they in particular take values in elements of $\mathcal{U}$.\\
Both these approaches are equally good in overcoming all the size issues that arise in this study. The rest of the article is written  following the terminology conventions of NBG set theory, so for example we will be talking about “proper classes” and “ordinals” instead of “large sets” and “$\mathcal{V}$-ordinals”, but the conversion from one to the other is immediate and virtually problem free.\\
In particular, the definition of $\infty$-category in \cite{HTT} (which I'll be using as the main reference for these objects) is a simplicial \emph{set} with properties. However, Lurie himself explains in section 1.2.15 that he is adopting the approach of TG set theory throughout his book, so that his sets are possibly large sets in the sense specified above. Therefore, the fact that most of our $\infty$-categories are in fact proper classes should not be of any serious concern, since it is still merely part of the correspondence between the NBG and the TG terminology.

\section{A review of localizations}
In this section we will review several flavors of localization, of its dual notion, and the related notion of cellularization. We will start with the simplest ones.

\begin{defn}
Let $\mathcal{C}$ be a 1-category, and let $S$ be a class of morphisms in $\mathcal{C}$. An external localization of $\mathcal{C}$ at $S$ is a category $\mathcal{C}[S^{-1}]$ with a functor $\gamma: \mathcal{C} \to \mathcal{C}[S^{-1}]$ having the following universal property: for each category $\mathcal{D}$, precomposition with $\gamma$ induces an equivalence of functor categories

\begin{center}
$\gamma^{\ast}: \Fun(\mathcal{C}[S^{-1}],\mathcal{D}) \to \Fun_S(\mathcal{C},\mathcal{D})$
\end{center}

where the codomain is the full subcategory spanned by all functors $\mathcal{C} \to \mathcal{D}$ taking morphisms of $S$ to isomorphisms.
\end{defn}

Since it is defined in terms of a universal property, it is clear that, whenever it exists, an external localization is unique up to equivalence of categories. Moreover, if size issues are ignored, external localizations always exist in some possibly larger Grothendieck universe (or, for those who prefer the language of sets and classes, the external localization might not be locally small). Historically, one of the most used explicit constructions of it is the Gabriel-Zisman localization, introduced in \cite{GabrielZisman}. It is promptly transferred to the $\infty$-categorical context.

\begin{defn}
Let $\mathcal{C}$ be an $\infty$-category, and let $S$ be a class of morphisms in $\mathcal{C}$. An external localization of $\mathcal{C}$ at $S$ is an $\infty$-category $\mathcal{C}[S^{-1}]$ with a functor $\gamma: \mathcal{C} \to \mathcal{C}[S^{-1}]$ having the following universal property: for each $\infty$-category $\mathcal{D}$, precomposition with $\gamma$ induces an equivalence of functor $\infty$-categories

\begin{center}
$\gamma^{\ast}: \Fun(\mathcal{C}[S^{-1}], \mathcal{D}) \to \Fun_S(\mathcal{C},\mathcal{D})$
\end{center}

where the codomain is the full subcategory spanned by all functors $\mathcal{C} \to \mathcal{D}$ taking morphisms of $S$ to equivalences.
\end{defn}

Similarly, this localization always exists in some Grothendieck universe, and it is unique up to a contractible choice. The most common explicit way of defining it that applies to a generic $\infty$-category, when represented by a quasi-category, is the following: consider the subcategory $\mathbf{S} \subseteq \mathcal{C}$ spanned by all morphisms of $S$, then take a pushout

\begin{center}
\begin{tikzcd}
\mathbf{S} \ar[r] \ar[d, hook] & F_K \mathbf{S} \ar[d]\\
\mathcal{C} \ar[r] & \bullet
\end{tikzcd}
\end{center}

in the category of simplicial sets, where $F_K$ is a Kan fibrant replacement. Our $\infty$-category of interest will be a Joyal fibrant replacement of the bottom right corner.\\

The next definition corresponds with the notion of orthogonality class in a category, and it is still often called that in the context of 1-categories. However, in order to compare and contrast it with the two definitions above, we will call it here internal localization.

\begin{defn}
Let $\mathcal{C}$ be a category ($\infty$-category), let $C$ be an object and let $f: A \to B$ be a morphism in $\mathcal{C}$. We say that $C$ and $f$ are orthogonal to each other if precomposition with $f$ induces a bijection of sets (homotopy equivalence of spaces)

\begin{center}
$f^{\ast}: \mathcal{C}(B,C) \to \mathcal{C}(A,C)$.
\end{center}
\end{defn}

\begin{defn}
Let $\mathcal{C}$ be a category ($\infty$-category), and let $S$ be a class of morphisms in $\mathcal{C}$. The internal localization of $\mathcal{C}$ at $S$ is the full subcategory $S^{-1} \mathcal{C}$ of all objects $C \in \mathcal{C}$ that are orthogonal to every morphism in $S$. Such objects are sometimes also called $S$-local.
\end{defn}


The notions of external and internal localization share the noun in their name because they pursue the same goal, that is, regarding the morphisms in $S$ as invertible, although they achieve this with two complementary approaches. In the external localization, one adds all the material (and nothing more) that is necessary in order to make the morphisms at issue invertible; in the internal localization, one selects the biggest possible environment in which they are already virtually invertible.\\
In general, these two constructions are not equivalent, and in particular the internal localization doesn't even satisfy any universal property. However, there is at least one case that we will be dealing with in which the two notions coincide and, in particular, the internal localization will also enjoy a universal property. It will be stated at the beginning of \Cref{section 3}.\\

Mostly, we will be interested in internal localizations, and part of our work will be basically investigating under what conditions we can obtain results with the same flavor as \Cref{intextinf} and \Cref{intext}, in particular when they are reflective subcategories. In order to do so, we will frequently resort to the language of model categories, whose purposes are homotopical but whose techniques are considerably more intricate than those of $\infty$-categories. We will therefore need model-categorical definitions of internal localizations, and a way to efficiently transfer information between model categories and $\infty$-categories. The former demand will be met by Bousfield localizations, the latter by the aforementioned external localizations.

\begin{defn}
\hspace{10pt} \newline
\begin{itemize}
\item Let $f: X \to Y$ be a morphism and let $Z$ be an object in a simplicial model category $\mathbf{M}$. We say that $f$ and $Z$ are homotopy orthogonal to each other if for every cofibrant replacement $QX \to QY$ of $f$ and every fibrant replacement $PZ$ of $Z$, the induced morphism

\begin{center}
$\mathbf{M}(QY,PZ) \to \mathbf{M}(QX,PZ)$
\end{center}
is a homotopy equivalence of simplicial sets.

\item Let $\mathbb{S}$ be a class of morphisms in a simplicial model category $\mathbf{M}$. We say that an object $Z \in \mathbf{M}$ is $\mathbb{S}$-local if it is homotopy orthogonal to every morphism in $\mathbb{S}$.

\item In the same setting as above, we say that a morphism $X \to Y$ is an $\mathbb{S}$-local equivalence if it is homotopy orthogonal to every $\mathbb{S}$-local object.
\end{itemize}
\end{defn}

\begin{defn}
Let $\mathbf{M}$ be a simplicial model category, and let $\mathbb{S} \subseteq \mathbf{M}$ be a class of morphisms. We say that a left Bousfield localization of $\mathbf{M}$ at $\mathbb{S}$ is another model category $L_{\mathbb{S}} \mathbf{M}$ having the same underlying category as $\mathbf{M}$, the same cofibrations, and $\mathbb{S}$-local equivalences as weak equivalences.
\end{defn}

\begin{rem}
Note that such a localization does not always exist, because it might fail to give rise to a closed model structure. However, whenever it does exist, it is obviously unique. Moreover, even in that case, it need not be simplicial, in that the new model structure might not satisfy the compatibility conditions with the simplicial one.
\end{rem}

Now, this definition only applies to the case of simplicial model categories. The following definition is more general and moreover it highlights the fact that Bousfield localizations satisfy a universal property.

\begin{defn}
Let $\mathbf{M}$ be a model category (not necessarily simplicial), and let $\mathbb{S}$ be a class of morphisms in $\mathbf{M}$. We say that a left Bousfield localization of $\mathbf{M}$ at $\mathbb{S}$, often denoted by $\mathbf{M} \backslash \mathbb{S}$, is a model structure on the same underlying category as $\mathbf{M}$ having the same cofibrations, and with the following property:

\begin{itemize}
\item $\mathbf{M} \backslash \mathbb{S}$ is universal among model categories $\mathbf{N}$ with a left Quillen functor $F: \mathbf{M} \to \mathbf{N}$ such that $F \circ Q$ takes morphisms in $\mathbb{S}$ to weak equivalences, where $Q$ is a cofibrant replacement.
\end{itemize}
\end{defn}

We conclude this section by briefly presenting, for the sake of completeness, the dual notions to many of the concepts so far introduced. In particular, external localizations do not have an interesting dualization, at least for our purposes.

\begin{defn}
Let $\mathcal{C}$ be a category ($\infty$-category), let $C$ be an object and let $f: A \to B$ be a morphism in $\mathcal{C}$. We say that $C$ and $f$ are co-orthogonal to each other if postcomposition with $f$ induces a bijection of sets (homotopy equivalence of spaces)

\begin{center}
$f_{\ast}: \mathcal{C}(C,A) \to \mathcal{C}(C,B)$.
\end{center}
\end{defn}

\begin{defn}
Let $\mathcal{C}$ be a category ($\infty$-category), and let $S$ be a class of morphisms in $\mathcal{C}$. The colocalization of $\mathcal{C}$ at $S$ is the full subcategory of all objects $C \in \mathcal{C}$ that are co-orthogonal to morphisms in $S$. Such objects are sometimes also called $S$-colocal.
\end{defn}

\begin{defn}
\hspace{10pt} \newline
\begin{itemize}
\item Let $f: X \to Y$ be a morphism and let $Z$ be an object in a simplicial model category $\mathbf{M}$. We say $f$ and $Z$ are homotopy co-orthogonal to each other if, for every fibrant replacement $PX \to PY$ of $f$ and every cofibrant replacement $QZ$ of $Z$, the induced morphism

\begin{center}
$\mathbf{M}(QZ,PX) \to \mathbf{M}(QZ,PY)$
\end{center}
is a homotopy equivalence of simplicial sets.

\item Let $\mathbb{S}$ be a class of morphisms in a simplicial model category $\mathbf{M}$. We say that an object $Z \in \mathbf{M}$ is $\mathbb{S}$-colocal if it is homotopy co-orthogonal to every morphism in $\mathbb{S}$.

\item In the same setting as above, we say that a morphism $X \to Y$ is an $\mathbb{S}$-colocal equivalence if it is homotopy co-orthogonal to every $\mathbb{S}$-colocal object.
\end{itemize}
\end{defn}

\begin{defn}
Let $\mathbf{M}$ be a simplicial model category, and let $\mathbb{S} \subseteq \mathbf{M}$ be a class of morphisms. We say that a right Bousfield localization of $\mathbf{M}$ at $\mathbb{S}$ is another model category $R_{\mathbb{S}} \mathbf{M}$ having the same underlying category as $\mathbf{M}$, the same fibrations, and $\mathbb{S}$-colocal equivalences as weak equivalences.
\end{defn}

\begin{defn}
Let $\mathbf{M}$ be a model category (not necessarily simplicial), and let $\mathbb{S}$ be a class of morphisms in $\mathbf{M}$. We say that a right Bousfield localization of $\mathbf{M}$ at $\mathbb{S}$, often denoted as $\mathbf{M} \slash \mathbb{S}$, is a model structure on the same underlying category as $\mathbf{M}$ having the same fibrations, and with the following property:

\begin{itemize}
\item $\mathbf{M} \slash \mathbb{S}$ is universal among model categories $\mathbf{N}$ with a right Quillen functor $G: \mathbf{N} \to \mathbf{M}$ such that $G \circ P$ takes morphisms in $\mathbb{S}$ to weak equivalences, where $P$ is a fibrant replacement.
\end{itemize}
\end{defn}

In conclusion, we give an additional definition that is relevant for us only in one of its dual forms, due to an asymmetry between the sort of classes with respect to which it is possible to localize or colocalize combinatorial model categories.

\begin{defn} \label{cellularization}
Let $\mathbf{M}$ be a simplicial model category, and let $\mathbf{D}$ be a class of objects in $\mathbf{M}$. We say that a cellularization of $\mathbf{M}$ at $\mathbf{D}$ is a right Bousfield localization of $\mathbf{M}$ at $\mathbb{S}_{\mathbf{D}}$, where $\mathbb{S}_{\mathbf{D}}$ is the class of morphisms that are homotopy co-orthogonal to $\mathbf{D}$.
\end{defn}

\section{Vop\v{e}nka's principle(s)}
The majority of the new results presented in this article rely on the assumption the Vop\v{e}nka's principle, or its weaker version, is true. This section is dedicated to reminding the statements of these principles, and introduce higher categorical formulations of them. In the weaker case, we also provide a proof that the higher categorical formulation is equivalent to the classical one. In the stronger case, the equivalence of the two statements remains an open problem.

\begin{prop}[\cite{LocPresAccCat}, 6.3] \label{vopenka}
The following statements are equivalent:

\begin{enumerate}
\item In a presentable category, there can't be a proper class of objects such that the only morphisms between them are identities;

\item There is no full embedding of the category $\Ord$ in any presentable category;

\item Given objects $(A_i)_{i \in \Ord}$ in a presentable category, there exists a morphism $A_i \to A_j$ for some ordinals $i < j$.
\end{enumerate}
\end{prop}

\begin{defn}[$\vp$]
We say that Vop\v{e}nka's principle holds if the equivalent conditions of \Cref{vopenka} are assumed true.
\end{defn}

In particular, condition (2) of \Cref{vopenka} seems the most formal, and therefore the easiest one to transpose to a higher categorical context. Moreover, the following definition of weak Vop\v{e}nka's principle is formulated in a very similar way, which makes using condition (2) as a definition of Vop\v{e}nka's principle even more appealing.

\begin{defn}[$\wvp$]
We say that weak Vop\v{e}nka's principle holds if for any presentable category $\mathcal{K}$ there is no full embedding $\Ord^{op} \hookrightarrow \mathcal{K}$.
\end{defn}

Moreover, there is an apparently more relaxed formulation of weak Vop\v{e}nka's principle, which has recently been proven to be equivalent to it.

\begin{defn}[$\swvp$]
We say that semi-weak Vop\v{e}nka's principle holds if, given any presentable category $\mathcal{K}$, there is no $\Ord$-indexed sequence of objects of $\mathcal{K}$ such that $\mathcal{K}(K_{\alpha},K_{\beta})$ is empty for $\alpha < \beta$ and non-empty for $\alpha \geq \beta$.
\end{defn}

By \Cref{vopenka}, it is clear that $\vp \Rightarrow \swvp$ and, moreover, it is also obvious that $\swvp \Rightarrow \wvp$. Here is the set-theoretical collocation of the three above statements: the consistency of all three of them follows from the assertion that there exists a huge cardinal (\cite{LocPresAccCat}, Theorem A.18) and they imply the assertion that there are arbitrarily large measurable cardinals (\cite{LocPresAccCat}, Corollary A.7). In particular, this also implies that they are independent of ZFC.\\
For what concerns their interactions with one another, the next two results clarify the whole picture of implications between the three statements.

\begin{prop}[\cite{VP-WVP}, 1.4] \label{wvpswvp}
Weak Vop\v{e}nka's principle is equivalent to semi-weak Vop\v{e}nka's principle.
\end{prop}

\begin{prop}[\cite{VP-WVP}, 1.2]
Weak Vop\v{e}nka's principle is strictly weaker than Vop\v{e}nka's principle, i.e., there is a model of ZFC in which the former but not the latter is satisfied.
\end{prop}

Now we have a complete picture of the 1-dimensional Vop\v{e}nka's principles, and we are ready to introduce the higher dimensional ones, investigating their relations to one another and to their 1-dimensional versions.

\begin{defn}[$\ivp$]
We say that $\infty$-Vop\v{e}nka's principle holds if for any presentable $\infty$-category $\mathcal{K}$ there is no full embedding $\N(\Ord) \hookrightarrow \mathcal{K}$.
\end{defn}

\begin{defn}[$\iwvp$]
We say that weak $\infty$-Vop\v{e}nka's principle holds if for any presentable $\infty$-category $\mathcal{K}$ there is no full embedding $\N(\Ord)^{op} \hookrightarrow \mathcal{K}$.
\end{defn}

\begin{defn}[$\iswvp$]
We say that semi-weak $\infty$-Vop\v{e}nka's principle holds if, given any presentable $\infty$-category $\mathcal{K}$, there is no $\Ord$-indexed sequence of objects of $\mathcal{K}$ such that $\mathcal{K}(K_{\alpha},K_{\beta})$ is empty for $\alpha < \beta$ and non-empty for $\alpha \geq \beta$.
\end{defn}

We will now show that the relations running between these three statements are exactly the same as for their 1-dimensional versions, that is, $\ivp$ is strictly stronger than $\iwvp$ and $\iswvp$ which are equivalent. Moreover, we will also prove that $\iwvp$ is equivalent to $\wvp$. Unfortunately, the question whether $\ivp$ is equivalent to $\vp$ remains open, although it is easily observed that $\ivp \Rightarrow \vp$, since a counterexample to $\vp$ gives rise to a counterexample of $\ivp$ by taking nerves.

\begin{prop} \label{swvpiswvp}
$\swvp$ is equivalent to $\iswvp$.
\end{prop}

\begin{proof}
Since presentability of ordinary categories and inhabitedness of mapping sets are checked on nerves, if $\Ord^{op} \to \mathcal{C}$ is a counterexample of $\swvp$ then the induced $\N(\Ord^{op}) \to \N(\mathcal{C})$ is a counterexample of $\iswvp$, and we have the implication $\iswvp \Rightarrow \swvp$.\\
Coversely, assume that there is a counterexample of $\iswvp$, that is, a presentable $\infty$-category $\mathcal{K}$ containing objects $(K_{\alpha})_{\alpha \in \Ord}$ such that there is at least one morphism $K_{\alpha} \to K_{\beta}$ whenever $\alpha \geq \beta$, and there are no such morphisms when $\alpha < \beta$. By \cite{HTT}, Proposition A.3.7.6, there is a combinatorial simplicial model category $\mathbf{M}$ such that $\N_{\Delta}(\mathbf{M}^{\circ}) \simeq \mathcal{K}$. In particular, the objects $(K_{\alpha})_{\alpha \in \Ord}$ can be regarded as objects of $\mathbf{M}^{\circ}$, with the property above. Forgetting both the simplicial and the model structure on $\mathbf{M}$ gives a presentable category $\mathcal{M}$ containing objects $(K_{\alpha})$ still having the same property, which contradicts $\swvp$ and thus concludes our proof.
\end{proof}

\begin{prop} \label{wvpiwvp}
$\wvp$ is equivalent to $\iwvp$.
\end{prop}

\begin{proof}
As above, a counterexample to $\wvp$ gives rise to a counterexample to $\iwvp$ by taking nerves.\\
Conversely, suppose there is a counterexample to $\iwvp$. In particular, it also contradicts $\iswvp$, therefore by \Cref{swvpiswvp}, there is a counterexample to $\swvp$. Finally, by \Cref{wvpswvp} we can fabricate a counterexample to $\wvp$, which concludes our proof.
\end{proof}

\begin{cor} \label{iwvpiswvp}
$\iwvp$ is equivalent to $\iswvp$.
\end{cor}

Just for the purpose of recording, we also state the following trivial result.

\begin{prop} \label{vpivp}
$\ivp$ implies $\vp$.
\end{prop}

\begin{cor}
$\ivp$ implies $\iwvp$.
\end{cor}

Although it is not know if \Cref{vpivp} admits a converse, the formulation of $\vp$ given in \Cref{vopenka}(1) does have an equivalent $\infty$-categorical statement. We will now show this.

\begin{lem} \label{altvopenka}
The following statements are equivalent:

\begin{enumerate}
\item In a presentable category, there can't be a proper class of objects such that the only morphisms between them are identities;
\item In a presentable category, there can't be a proper class of objects such that the only morphisms between them are endomorphisms.
\end{enumerate}
\end{lem}

\begin{proof}
This follows from the equivalence $(1) \Leftrightarrow (3)$ of \Cref{vopenka}.
\end{proof}

\begin{prop} \label{altvopenkaivopenka}
The following statements are equivalent:

\begin{enumerate}
\item In a presentable category $\mathcal{C}$, there can't be a proper class of objects in which $\mathcal{C}(X,Y)$ is empty when $X \neq Y$ and a singleton when $X=Y$;
\item In a presentable $\infty$-category $\mathcal{K}$, there can't be a proper class of objects in which $\mathcal{K}(X,Y)$ is empty when $X \neq Y$ and contractible when $X=Y$.
\end{enumerate}
\end{prop}

\begin{proof}
A counterexample to (1) immediately gives a counterexample to (2) by taking the nerve.\\
Suppose we have a counterexample to (2), that is, we are given a presentable $\infty$-category $\mathcal{K}$ with a proper class of objects in which $\mathcal{K}(X,Y)$ is empty whenever $X \neq Y$ and a contractible space whenever $X=Y$. By \cite{HTT}, Proposition A.3.7.6, there is a combinatorial simplicial model category $\mathbf{M}$ such that $\N_{\Delta}(\mathbf{M}^{\circ}) \simeq \mathcal{K}$. In particular, we have a proper class of objects $(K_i)$ in $\mathbf{M}^{\circ}$ such that $\mathbf{M}(X,Y)$ is empty whenever $X \neq Y$. Forgetting both the simplicial and the model structure on $\mathbf{M}$ gives a presentable category $\mathcal{M}$ having the same property, which contradicts \Cref{altvopenka}(2), therefore also contradicts \Cref{altvopenka}(1), which is the same as (1) in the present proof.
\end{proof}

\section{Internal localizations under Vop\v{e}nka's principle(s)} \label{section 3}
In this section, we will review some results that are already known concerning internal localizations and colocalizations and Bousfield localizations.

\begin{prop}[\cite{LocPresAccCat}, 1.39, 1.40] \label{smallorthogonalityreflective}
Let $S$ be a small set of morphisms in a presentable category $\mathcal{K}$. Then the internal localization of $\mathcal{K}$ at $S$ is a reflective subcategory and it is itself presentable.
\end{prop}

\begin{prop}[\cite{HTT}, 5.5.4.15] \label{infsmallorthogonalityreflective}
Let $\mathcal{C}$ be a presentable $\infty$-category, and let $S$ be a small set of morphisms in $\mathcal{C}$. Then the internal localization $S^{-1} \mathcal{C}$ is a reflective subcategory, and moreover it is itself presentable.
\end{prop}

The following two propositions express some interesting interaction between the internal and external localizations taken at the same class of morphisms.

\begin{prop}[First half of the proof of \cite{HTT}, 5.5.4.20] \label{intextinf}
Let $\mathcal{C}$ be a presentable $\infty$-category, and let $S$ be a small set of morphisms in $\mathcal{C}$. Then the internal localization of $\mathcal{C}$ at $S$ satisfies the following universal property: for every $\infty$-category $\mathcal{D}$, the reflection functor $\mathcal{C} \to S^{-1} \mathcal{C}$ induces an equivalence of $\infty$-categories

\begin{center}
$\Fun^L(S^{-1} \mathcal{C}, \mathcal{D}) \to \Fun^L_S (\mathcal{C}, \mathcal{D})$
\end{center}

where the superscript $L$ restricts to colimit-preserving functors, and the subscript $S$ restricts to functors sending all elements of $S$ to equivalences.
\end{prop}

\begin{rem}
There are two minor mistakes in the proof of the cited Proposition 5.5.4.20 of \cite{HTT}. Firstly, the reasoning there only lets us conclude that $\bar{\eta}: \Fun(S^{-1}\mathcal{C}, \mathcal{D}) \to \Fun(\mathcal{C}, \mathcal{D})$ (appearing in the mentioned proof) is fully faithful, rather than an equivalence. However, this is entirely enough for our and Lurie's purpose.\\
Secondly, we need to use Proposition 5.5.4.16 at the end of the proof, instead of 5.5.4.11, which is the reason why we are only looking at colimit-preserving functors.
\end{rem}

\begin{prop}
Let $\mathcal{C}$ be a presentable $\infty$-category, and let $S$ be a small set of morphisms in $\mathcal{C}$. Given the reflection $L: \mathcal{C} \to S^{-1} \mathcal{C}$, let $\bar{S}$ denote the class of morphisms that are taken to equivalences by $L$. Then the internal localization of $\mathcal{C}$ at $\bar{S}$ satisfies the universal property of the corresponding external localization, i.e.

\begin{center}
$\Fun(S^{-1} \mathcal{C}, \mathcal{D}) \simeq \Fun_{\bar{S}} (\mathcal{C}, \mathcal{D})$
\end{center}

for every $\infty$-category $\mathcal{D}$. In particular, there is an equivalence of $\infty$-categories

\begin{center}
$S^{-1} \mathcal{C} \simeq \mathcal{C}[\bar{S}^{-1}]$.
\end{center}
\end{prop}

\begin{proof}
Observe first that the internal localization $S^{-1} \mathcal{C}$ is equivalent to $\bar{S}^{-1} \mathcal{C}$, as it readily follows by the definitions. The proof now retraces the proof of \cite{HTT}, Proposition 5.5.4.20, except that it is slightly simplified in the end. Namely, instead of $S$ we use $\bar{S}$, which is always strongly saturated by \cite{HTT}, Proposition 5.5.4.11, so that we don't need to restrict ourselves to colimit-preserving functors, and the statement holds more in general for the $\infty$-category of all functors $\mathcal{C} \to \mathcal{D}$.
\end{proof}

\begin{rem}
Since presentability of ordinary categories and adjunctions between them are detected on their nerves, we can alternatively obtain \Cref{smallorthogonalityreflective} as a corollary of \Cref{infsmallorthogonalityreflective}. Moreover, the next additional result immediately follows as well.
\end{rem}

\begin{cor} \label{intext}
Let $\mathcal{C}$ be a presentable category and let $\bar{S}$ be a strongly saturated class of morphisms in $\mathcal{C}$ that is of small generation. Then the internal localization at $\bar{S}$ is equivalent to the corresponding external localization.
\end{cor}

The following are known results that we will try to transfer to the $\infty$-world.

\begin{prop}[\cite{LocPresAccCat}, 6.24] \label{localizationreflective}
Assuming $\vp$, let $\mathcal{C}$ be a presentable category, and let $\mathcal{A}$ be a full subcategory. Then the following statements are equivalent:

\begin{enumerate}
\item $\mathcal{A}$ is closed under small limits in $\mathcal{C}$;
\item $\mathcal{A}$ is the internal localization at a class of morphisms $S$;
\item $\mathcal{A}$ is the internal localization at a set of morphisms $S$;
\item $\mathcal{A}$ is reflective in $\mathcal{C}$.
\end{enumerate}

Moreover, each of these properties implies presentability of $\mathcal{A}$.
\end{prop}

\begin{rem}
In the above proposition, the implications $(3) \Rightarrow (2) \Rightarrow (1)$ are trivial and always true regardless of any set-theoretical assumptions. The implication $(3) \Rightarrow (4)$ is taken care of by \Cref{smallorthogonalityreflective}, and the implication $(4) \Rightarrow (2)$ is also verified by taking the class $S$ of all reflections. The results really depending on set-theory are then $(1) \Rightarrow (4)$, in fact only depending on $\wvp$, and $(2) \Rightarrow (3)$ (or, alternatively, $(4) \Rightarrow (3)$, which is the strategy adopted in \cite{LocPresAccCat}).
\end{rem}

Our next goal will be to generalize \Cref{localizationreflective} to $\infty$-categories. Observe that the implications $(3) \Rightarrow (2) \Rightarrow (1)$ are likewise trivial, and $(3) \Rightarrow (4)$ is precisely \Cref{infsmallorthogonalityreflective}. Therefore, it only remains to show $(1) \Rightarrow (4)$ and $(2) \Rightarrow (3)$. 

\begin{lem} \label{smallundercategory}
Let $\mathcal{C}$ be a locally small $\infty$-category, and $\mathcal{A} \subseteq \mathcal{C}$ an essentially small full subcategory. Then for each object $C \in \mathcal{C}$ the undercategory $\mathcal{A}_{C/}$ is essentially small.
\end{lem}

\begin{proof}
Observe that there is a coherently commutative diagram

\begin{center}
\begin{tikzcd}
\mathcal{A}_{C/} \ar[d] \ar[r] & \mathcal{C}_{C/} \ar[d] \ar[r] & \mathcal{S}_{\ast /} \ar[d]\\
\mathcal{A} \ar[r] & \mathcal{C} \ar[r] & \mathcal{S}
\end{tikzcd}
\end{center}

in $\biginfcat$, where the left square is a pullback by definition, and the right square is so by the structure theory of fibrations (for a treatment thereof, see for example \cite{HTT}, Section 3.3, and in particular Proposition 3.3.2.5). Moreover, the classifying functor $\mathcal{C} \to \mathcal{S}$ sends each object $D$ to the mapping space $\mathcal{C}(C,D)$. By local smallness of $\mathcal{C}$ and essential smallness of $\mathcal{A}$, we may choose a cardinal $\kappa$ such that $\mathcal{C}(C,D)$ is essentially $\kappa$-small whenever $D \in \mathcal{A}$. It follows that the lower composition in the diagram factors through $\mathcal{S}_{\kappa}$, and therefore there is another diagram

\begin{center}
\begin{tikzcd}
\mathcal{A}_{C/} \ar[d] \ar[r] & (\mathcal{S}_{\kappa})_{\ast /} \ar[d] \ar[r] & \mathcal{S}_{\ast /} \ar[d]\\
\mathcal{A} \ar[r] & \mathcal{S}_{\kappa} \ar[r] & \mathcal{S}.
\end{tikzcd}
\end{center}

Here the right square is clearly a pullback, and so is the outer square by the diagram above. By the pasting law, it follows that the left square is a pullback. Since $\mathcal{A}_{C/}$ is now seen to be a pullback of essentially small $\infty$-categories, it is itself essentially small.
\end{proof}

For the next lemma, recall that a subcategory $\mathcal{A} \subseteq \mathcal{C}$ is said to be weakly reflective if for every $C \in \mathcal{C}$ the undercategory $\mathcal{A}_{C/}$ has a weakly initial set.

\begin{lem} \label{prodweaklyreflective}
Assuming $\wvp$, then a full subcategory of a presentable $\infty$-category which is closed under products is weakly reflective.
\end{lem}

\begin{proof}
By \Cref{swvpiswvp} and \Cref{wvpswvp}, $\wvp$ is equivalent to $\iswvp$. Let $\mathcal{C}$ be a presentable $\infty$-category and let $\mathcal{A} \subseteq \mathcal{C}$ be a full subcategory that is closed under products. Express $\mathcal{A}$ as a union of a chain of small subcategories

\begin{center}
$\mathcal{A} = \bigcup_{\alpha \in \Ord} \mathcal{A}_{\alpha}$.
\end{center}

For each $\alpha$, denote with $\overline{\mathcal{A}_{\alpha}}$ the closure of $\mathcal{A}_{\alpha}$ under products. Note that $\mathcal{A} = \bigcup_{\alpha \in \Ord} \overline{\mathcal{A}_{\alpha}}$. Let us first construct weak reflections of an object $C \in \mathcal{C}$ in $\overline{\mathcal{A}_{\alpha}}$. In other words, we need a weakly initial object in $(\overline{\mathcal{A}_{\alpha}})_{C/}$. Observe that by \Cref{smallundercategory} the $\infty$-category $(\mathcal{A}_{\alpha})_{C/}$ is essentially small. Without loss of generality, let us assume that it is strictly small. In particular, its set of objects will be small. Now we compute the product

\begin{center}
$C_{\alpha} = \prod_{A \in ((\mathcal{A}_{\alpha})_{C/})_0} A$.
\end{center}

There is a natural map $C \to C_{\alpha}$ which is induced by the maps $C \to A$ constituting the product itself. Moreover, $C_{\alpha} \in \overline{\mathcal{A}_{\alpha}}$ by construction, therefore the map $C \to C_{\alpha}$ may be regarded as an object of $(\overline{\mathcal{A}_{\alpha}})_{C/}$. We will show that it is weakly initial in there.\\
First, for $A \in (\mathcal{A}_{\alpha})_{C/}$, the projection from the product $C_{\alpha} \to A$ is compatible with $C \to C_{\alpha}$. It remains to show that the class of objects $C \to B$ such that there exists a factorization $C_{\alpha} \to B$ is closed under products.\\
If $B = \prod_{i \in I} B_i$ is a product of objects living under $C$, and each $C \to B_i$ factors through $C \to C_{\alpha}$, then the universal property of products produces a map $C_{\alpha} \to B$ and a 2-cell

\begin{center}
\begin{tikzcd}
& C \ar[dr] \ar[dl]\\
C_{\alpha} \ar[rr] && B.
\end{tikzcd}
\end{center}

So far, we have proven that for every ordinal $\alpha$ the subcategory $\overline{\mathcal{A}_{\alpha}}$ is weakly reflective. Now, if we can find an ordinal $\alpha$ such that for every $\beta \geq \alpha$ the map $C \to C_{\beta}$ factors through $C \to C_{\alpha}$, we will have obtained a weak reflection in all $\mathcal{A}$. Suppose by contradiction that this is not true. In that case, for every $\alpha$ we can find a $\beta > \alpha$ such that $C \to C_{\beta}$ does not factor through $C \to C_{\alpha}$. Now, for $\gamma \leq \alpha$ we have a factorization $C_{\alpha} \to C_{\gamma}$ by weak reflectivity, so we also have that $C \to C_{\beta}$ does not factor through $C \to C_{\gamma}$, because otherwise a composite $C_{\alpha} \to C_{\gamma} \to C_{\beta}$ would produce a factorization of $C \to C_{\beta}$ through $C \to C_{\alpha}$. Likewise, if $\alpha < \gamma < \beta$, the map $C \to C_{\beta}$ doesn't factor through $C \to C_{\gamma}$. To see this, observe that by assumption there is a $\beta' > \beta$ with the same property that $\beta$ has with respect to $\alpha$, so we may just repeat the argument above after replacing $\alpha$ by $\beta$ and $\beta$ by $\beta'$. To summarize, $C \to C_{\beta}$ does not factor through $C \to C_{\gamma}$ for any $\gamma < \beta$. Since we can choose such a $\beta$ for any $\alpha$, there is a proper class of ordinals $\mathbf{B}$ such that none of the arrows $C \to C_{\beta}$ factors through any of the $C \to C_{\alpha}$ for a lower $\alpha$. In particular, we have produced a class of objects $(C_{\alpha})_{\alpha \in \mathbf{B}}$ in $\mathcal{C}_{C/}$ (which is presentable by \cite{HTT}, Proposition 5.5.3.11) such that $\mathcal{C}_{C/}(C_{\alpha}, C_{\beta})$ is non-empty for $\alpha \geq \beta$ by weak reflectivity, and empty for $\alpha < \beta$ by construction. Since $\mathbf{B} \cong \Ord$, this contradicts $\iswvp$.
\end{proof}

\begin{prop} \label{limclosurereflective}
Assuming $\wvp$, given a presentable $\infty$-category $\mathcal{C}$ and full subcategory $\mathcal{D} \subseteq \mathcal{C}$, if $\mathcal{D}$ is closed under limits then it is reflective.
\end{prop}

This statement is precisely the transposition of $(1) \Rightarrow (4)$ of \Cref{localizationreflective} to $\infty$-categories.

\begin{proof}
Since the inclusion functor $\mathcal{D} \hookrightarrow \mathcal{C}$ tautologically preserves small limits, by \cite{AFT}, Theorem 3.2.5 we are reduced to proving that for every object $C \in \mathcal{C}$ the undercategory $\mathcal{D}_{C/}$ has a weakly initial set. Now $\mathcal{D}$ is obviously closed under products, therefore the desired conclusion follows from \Cref{prodweaklyreflective}.
\end{proof}

\begin{cor} \label{smalllimclosurereflective}
Assuming $\wvp$, given a presentable $\infty$-category $\mathcal{C}$ and a full subcategory $\mathcal{D} \subseteq \mathcal{C}$, then the limit closure of $\mathcal{D}$ in $\mathcal{C}$ is reflective.
\end{cor}

\begin{rem} \label{smalllimclosureremark}
Note that in the first part of the proof of \cite{LocPresAccCat}, Theorem 6.22 it is proven that the limit closure of a \emph{small} subcategory is reflective, without assuming $\wvp$. However, I didn't find a way to adapt that part of the proof to the $\infty$-world, so the proof of my \Cref{limclosurereflective} rather relies on \Cref{prodweaklyreflective}, which needs $\wvp$ anyway. Therefore, it remains unknown if it is possible to drop $\wvp$ from the statement of \Cref{smalllimclosurereflective} in the case where $\mathcal{D}$ is small.
\end{rem}

We now complete our first task by proving that the implication $(2) \Rightarrow (3)$ also holds true in $\infty$-categories. First, we want to remind a well-known result about the existence of left Bousfield localizations of combinatorial model categories, and an expansion thereof under the assumption that $\vp$ be true.

\begin{thm}[\cite{HirschhornModelCategories}, 4.1.1] \label{hirschhornleft}
Let $\mathbf{M}$ be a left proper combinatorial model category, and let $\mathbb{S}$ be a set of morphisms in $\mathbf{M}$. Then the left Bousfield localization of $\mathbf{M}$ at $\mathbb{S}$ exists and it is combinatorial and left proper. Moreover, if $\mathbf{M}$ is simplicial, then so is such localization.
\end{thm}

\begin{lem}[\cite{homotopyorthogonalsubcategory}, 1.3] \label{orthogonal1.3}
Assuming $\vp$, let $\mathcal{D}$ be any class of objects in a combinatorial simplicial model category $\mathbf{M}$, and let $\mathcal{S} = \mathcal{D}^{\bot}$ be the class of morphisms that are homotopy orthogonal to $\mathcal{D}$. Then there exists a set of maps $\mathcal{X}$ such that $\mathcal{X}^{\bot} = \mathcal{S}^{\bot}$.
\end{lem}

The next theorem appears, with two different proof strategies, as \cite{homotopyorthogonalsubcategory}, Theorem 1.4, and \cite{RosickyTholenleft-determined}, Theorem 2.3. We will spell out the former proof, based on \Cref{orthogonal1.3}, because we will later on resort to an analogous strategy in the dual case.

\begin{thm} \label{leftbousfieldvopenka}
Assuming $\vp$, then a left Bousfield localization at any class of morphisms $\mathbb{S}$ in a left proper combinatorial model category $\mathbf{M}$ exists and it is combinatorial and left proper. Moreover, if $\mathbf{M}$ is simplicial, then so is such localization.
\end{thm}

\begin{proof}
Let $\mathcal{D} := \mathbb{S}^{\bot}$, and $\mathcal{S} := \mathcal{D}^{\bot} = (\mathbb{S}^{\bot})^{\bot}$. Then we have

\begin{center}
$\mathcal{S}^{\bot} = (\mathcal{D}^{\bot})^{\bot} = ((\mathbb{S}^{\bot})^{\bot})^{\bot} = \mathbb{S}^{\bot}$.
\end{center}

By \Cref{orthogonal1.3}, there is a set of maps $\mathcal{X}$ such that $\mathcal{X}^{\bot} = \mathcal{S}^{\bot} = \mathbb{S}^{\bot}$. In other words, the left Bousfield localization generated by $\mathbb{S}$ corresponds with that generated by $\mathcal{X}$, which exists and is combinatorial, simplicial and left proper by \Cref{hirschhornleft}.
\end{proof}

\begin{prop} \label{reflectivethensmallorth}
Assume $\vp$, and let $\mathcal{C}$ be a presentable $\infty$-category. Then every internal localization of $\mathcal{C}$ at a class of morphisms can be expressed as the internal localization at a set of morphisms.
\end{prop}

We would like to use \cite{HTT}, Proposition 5.5.4.2(3), in particular we want to use (i) to obtain (iii). We will prove the stronger result that the internal localization is presentable. Unfortunately, in order to do so, we will have to reconstruct the whole $\mathcal{C}$ and its localization in a way that makes presentability of the latter referrable to previously known results. In particular we will exhibit the localization as a reflective subcategory, so incidentally this also constitutes an alternative proof of $(2) \Rightarrow (4)$ in the $\infty$-categorical version of \Cref{localizationreflective}.

\begin{proof}
Let $S$ be a class of morphisms in $\mathcal{C}$. Our goal, as mentioned right above, is to prove that $S^{-1} \mathcal{C}$ is presentable. As a side consequence of the strategy adopted, it will be automatically clear that it is a reflective subcategory, too.\\
By \cite{HTT}, Proposition A.3.7.6, there exists a combinatorial simplicial model category $\mathbf{M}$ and an equivalence $\mathcal{C} \simeq \N_{\Delta}(\mathbf{M}^{\circ})$. Moreover, if we follow the constrution given there, and observe that the injective model structure is left proper, we obtain that $\mathbf{M}$ is left proper, too (by \cite{HTT}, Proposition A.3.7.3). By \Cref{leftbousfieldvopenka}, if $\mathbb{S} \subseteq \mathbf{M}$ is the class of morphisms that give rise to morphisms of $S$ through the equivalence above, then the left Bousfield localization of $\mathbf{M}$ at $\mathbb{S}$ exists and it is combinatorial and simplicial. Since the underlying category of $L_{\mathbb{S}} \mathbf{M}$ is the same as the underlying category of $\mathbf{M}$, the adjunction (consisting of identities in both directions) presenting it as a Bousfield localization of $\mathbf{M}$ is simplicial. This adjunction then induces one between $\infty$-categories after being pulled through the simplicial nerve (\cite{HTT}, Proposition 5.2.4.6), so that we obtain an adjunction

\begin{center}
$\N_{\Delta} (L_{\mathbb{S}} \mathbf{M}^{\circ}) \leftrightarrows \N_{\Delta} (\mathbf{M}^{\circ}) \simeq \mathcal{C}$.
\end{center}

An inspection of mapping spaces reveals that the right adjoint is fully faithful. Moreover, since $L_{\mathbb{S}} \mathbf{M}$ is combinatorial and simplicial, a further application of \cite{HTT}, Proposition A.3.7.6 shows that $\N_{\Delta} (L_{\mathbb{S}} \mathbf{M}^{\circ})$ is presentable, therefore the only point left to prove is that $\N_{\Delta} (L_{\mathbb{S}} \mathbf{M}^{\circ}) \simeq S^{-1} \mathcal{C}$ or, equivalently, that an object $X \in \mathbf{M}^{\circ}$ is fibrant in the localized model structure if and only if it is $S$-local as an object of $\mathcal{C}$.\\
For this, observe that the former condition is equivalent to $X$ being $\mathbb{S}$-local by \cite{HTT}, Proposition A.3.7.3. Observe now that the domains and codomains of all morphisms in $\mathbb{S}$ are bifibrant by construction, therefore the morphisms in $S$ exhaust all $\mathbb{S}$, so that another simple inspection of mapping spaces yields that, for bifibrant objects of $\mathbf{M}$, $S$-localness is equivalent to $\mathbb{S}$-localness.
\end{proof}

\begin{rem} \label{vpandivp}
Although the proposition above concerns $\infty$-categorical objects, the 1-dimensional version of Vop\v{e}nka's principle is a sufficient hypothesis to prove it. Considering that, in the 1-dimensional world, the coimplication $(2) \Leftrightarrow (3)$ is equivalent to $\vp$, this seems to be a good reason to believe that in fact $\vp \Leftrightarrow \ivp$, although for the moment it remains conjectural.
\end{rem}

\section{Cellularizations under Vop\v{e}nka's principle(s)} \label{sectioncellularizations}
In this section we will focus on the dual notion of internal localizations, that is, internal colocalizations. First, we recall a 1-categorical result stated in \cite{LocPresAccCat}, Corollary 6.29, making it a partial analog to \Cref{localizationreflective}.

\begin{prop} \label{colocalizationcoreflective}
Assuming $\vp$, let $\mathcal{C}$ be a presentable category, and let $\mathcal{A}$ be a full subcategory. Then the following statements are equivalent:

\begin{enumerate}
\item $\mathcal{A}$ is closed under small colimits in $\mathcal{C}$;
\item $\mathcal{A}$ is the colocalization at a class of morphisms $S$;
\item $\mathcal{A}$ is coreflective in $\mathcal{K}$.
\end{enumerate}

Moreover, each of these properties implies presentability of $\mathcal{A}$.
\end{prop}

\begin{ex}
Even under $\vp$, it is not true that every colocalization of a presentable category can be expressed as the colocalization at a small set of morphisms, as demonstrated by the following example, provided to me by Ji\v{r}\'{i} Rosick\'{y} through personal communication.\\
Consider the class of all morphisms in $\Grp$. An object $G$ in its localization is co-orthogonal to the map $\mathbf{0} \to G$, so the identity $\id_G$ is equal to the zero morphism, which means that $G = \mathbf{0}$, and the colocalization at issue is simply $\{ \mathbf{0} \}$. We want to show that there is no set of morphisms exhibiting it as the corresponding co-orthogonality class.\\
Suppose by contradiction that such a set $S$ exists, then take $\kappa$ to be a cardinal bigger than the cardinalities of all domains and codomains in it. We can construct simple groups of arbitrarily large cardinality, so consider such a group $H$. For each diagram

\begin{center}
\begin{tikzcd}
& A \ar[d]\\
H \ar[ur, dotted] \ar[r] & B
\end{tikzcd}
\end{center}

the horizontal map is necessarily the zero morphism, because $H$ is simple and larger than $B$. Therefore, we can choose the zero morphism as the dotted arrow, and this choice is unique because $H$ is also bigger than $A$. So $H$ belongs to our small co-orthogonality class, which is a contradiction.
\end{ex}

Our strategy for investigating colocalizations in the $\infty$-world will be to make use of the dual notion of left Bousfield localizations, that is, of course, right Bousfield localizations. Although there is a well known existence result for right Bousfield localizations in combinatorial model categories, the fact that the definition of these is not self-dual entails a disparity between this and the existence result for left Bousfield localizations. Namely, while we normally left localize simply at small sets of morphisms, the following theorem tells us that the classes at which we can right localize are slightly more complicated: they are the classes of co-local morphisms with respect to small sets of \emph{objects}. In this case, the resulting localization is also known as the cellularization at a set of objects. Below, we will first show that under $\vp$ we can perform cellularization at arbitrarily large classes of objects, and this leads to the further consequence that cellularizations are not different from general right Bousfield localizations.

\begin{thm}[\cite{HirschhornModelCategories}, 5.1.1]
Let $\mathbf{M}$ be a right proper combinatorial model category, and let $\mathbf{K}$ be a set of objects in $\mathbf{M}$. Calling $\mathbb{S}_{\mathbf{K}}$ the class of morphisms that are homotopy co-orthogonal to $\mathbf{K}$, then the right Bousfield localization of $\mathbf{M}$ at $\mathbb{S}_{\mathbf{K}}$ exists and it is combinatorial and right proper. Moreover, if $\mathbf{M}$ is simplicial, then so is such localization.
\end{thm}

Unfortunately, the assumption of being right proper is an unreasonable one for our purposes, because every presentable $\infty$-category can be shown to underlie a combinatorial simplicial model category which is left, but not right proper. It is for this reason that we need the technology of semi-model categories: Bousfield localizations are still possible even without any properness assumption, as it is made explicit by the following theorem. We will only formulate it for right semi-model categories, but a proof in the case of weak model categories is given in \cite{Combinatorialweakmodel}, Remark 7.10 + Theorem 7.11. We need one preliminary definition, a generalization to homotopy (co-)orthogonality to semi-model categories. It is given in \cite{Barwickleftright}, Notation 3.61 in terms of hammock localizations, but any rephrasing in terms of mapping spaces in the underlying $\infty$-category yields of course an equivalent definition.

\begin{defn}
Let $\mathbf{M}$ be a weak or semi-model category, let $X$ be an object and $f: A \to B$ a morphism in $\mathbf{M}$. We say that $X$ and $f$ are homotopy co-orthogonal if the induced map

\begin{center}
$\mathbf{M}^{\Delta}(X,A) \to \mathbf{M}^{\Delta}(X,B)$
\end{center}

is a homotopy equivalence of simplicial sets, where $\mathbf{M}^{\Delta}(-,-)$ denotes the mapping space in the underlying $\infty$-category.\\
We also say that $X$ is $f$-colocal, or that $f$ is an $X$-colocal equivalence.
\end{defn}

\begin{thm}[\cite{Barwickleftright}, 5.6, 5.13, 5.16, 5.19] \label{weakrightbousfield}
Let $\mathbf{M}$ be a combinatorial right semi-model category, and let $\mathbf{K}$ be a set of objects in $\mathbf{M}$. Calling $\mathbb{S}_{\mathbf{K}}$ the class of morphisms that are homotopy co-orthogonal to $\mathbf{K}$, then the right Bousfield localization $R_{\mathbb{S}_{\mathbf{K}}} \mathbf{M}$ of $\mathbf{M}$ at $\mathbb{S}_{\mathbf{K}}$ exists and it is combinatorial. Moreover, a bifibrant object of $\mathbf{M}$ is cofibrant in $R_{\mathbb{S}_{\mathbf{M}}} \mathbf{M}$ if and only if it is $\mathbf{K}$-colocal.
\end{thm}

\begin{lem} \label{coorthogonality1.2}
Given a combinatorial right semi-model category $\mathbf{M}$ such that the underlying category of $\mathbf{M}$ is $\lambda$-presentable and there are a generating sets of cofibrations and trivial cofibrations between $\lambda$-compact objects, then the full subcategory of $\mathbf{M}^{[\mathbf{1}]}$ spanned by all weak equivalences between fibrant objects is closed under $\lambda$-filtered colimits.
\end{lem}

\begin{proof}
Since there is a generating set of trivial cofibrations between $\lambda$-compact objects, the class of fibrant objects is stable under $\lambda$-filtered colimits.\\
Let $I$ be a set of generating cofibrations between $\lambda$-compact objects. For $i: A \to B$ in $I$, define a functor $F_i: \mathbf{M}^{[\mathbf{1}]} \to \Set^{[\mathbf{1}]}$ sending a map $f: X \to Y$ to the map

\begin{center}
$\mathbf{M}(Y,A) \to \mathbf{M}(X,A) \times_{\mathbf{M}(X,B)} \mathbf{M}(Y,B)$
\end{center}

induced by $i$. Since $A$ and $B$ are $\lambda$-compact, this functor preserves $\lambda$-filtered colimits.\\
The full subcategory of $\mathbf{M}^{[\mathbf{1}]}$ spanned by all morphisms that have the right liting property against $i$ is precisely $F_i^{-1}(\catname{Epi})$, hence it is closed under $\lambda$-filtered colimits. It follows that the class of morphisms with fibrant codomain that lift against $i$ is closed under $\lambda$-filtered colimits.\\
The class of all trivial fibrations between fibrant objects is precisely the intersection of all these classes, ranging over $i \in I$, and is therefore itself closed under $\lambda$-filtered colimits.\\
To conclude, observe that the small object argument applied to $I$ yields a factorization of every morphism between fibrant objects as a cofibration followed by a trivial fibration, and this factorization can moreover be made into a $\lambda$-filtered colimit preserving functor. Calling $T$ the latter component of this functor, the class of weak equivalences between fibrant objects is precisely the preimage under $T$ of the class of trivial fibrations between fibrant objects. Since this latter class is closed under $\lambda$-filtered colimits, this concludes the proof.
\end{proof}




The following is an analog of \Cref{orthogonal1.3}, dealing with co-orthogonality instead of orthogonality, and exchanging the roles of maps and objects. It appears as \cite{Chornycellularization}, Lemma 1.3 for model categories, the present version is adapted for right semi-model categories.

\begin{lem} \label{coorthogonality1.3}
Assuming $\vp$, let $\mathcal{S}$ be any class of morphisms in a combinatorial right semi-model category $\mathbf{M}$, and let $\mathcal{D} = \mathcal{S}^{\top}$ be the class of objects that are homotopy co-orthogonal to $\mathcal{S}$. Then there exists a set of objects $\mathcal{X}$ such that $\mathcal{X}^{\top} = \mathcal{D}^{\top}$.
\end{lem}

\begin{proof}
The proof of \cite{Chornycellularization}, Lemma 1.3 adapts step by step, the only non-trivial point being verifying that $\mathcal{D}$ and $(\mathcal{X}^{\top})^{\top}$ are closed under $\lambda$-filtered colimits. For this, observe that the proof of \cite{Chornycellularization}, Lemma 1.2 readily adapts when all the objects in consideration are fibrant (weak equivalences between fibrant objects are closed under $\lambda$-filtered colimits by \Cref{coorthogonality1.2}). For the general case, observe that every object in a right semi-model category is weakly equivalent to a fibrant object, and our definition of homotopy co-orthogonality is clearly invariant under weak equivalence.
\end{proof}

\begin{prop} \label{cellularizationvopenka}
Assuming $\vp$, then the cellularization at any class of objects $\mathbf{K}$ in a combinatorial right semi-model category $\mathbf{M}$ exists and it is combinatorial.
\end{prop}

This theorem is a version of \cite{Chornycellularization}, Theorem 1.4 that dispenses with the assumption of right properness. The proof is essentially unchanged.

\begin{proof}
Let $\mathbb{S}_{\mathbf{K}}:= \mathbf{K}^{\top}$ be the class of morphisms that are homotopy co-orthogonal to $\mathbf{K}$. We have to show that the right Bousfield localization at $\mathbb{S}_{\mathbf{K}}$ exists as a weak model category.\\
Let us define $\mathcal{D} := \mathbb{S}_{\mathbf{K}}^{\top}$, then compute

\begin{center}
$\mathcal{D}^{\top} = (\mathbb{S}_{\mathbf{K}}^{\top})^{\top} = ((\mathbf{K}^{\top})^{\top})^{\top} = \mathbf{K}^{\top} = \mathbb{S}_{\mathbf{K}}$.
\end{center}

By \Cref{coorthogonality1.3}, there is a set of objects $\mathcal{X}$ such that $\mathcal{X}^{\top} = \mathbb{S}_{\mathbf{K}}$, so that the desired right Bousfield localization exists by \Cref{weakrightbousfield}.
\end{proof}

Under the assumption of $\vp$, cellularizations are not different from general right Bousfield localizations, or colocalizations, as the following statement should clarify.

\begin{cor} \label{rightbousfieldvopenka}
Assuming $\vp$, then the right Bousfield localization at any class of morphisms $\mathcal{S}$ in a combinatorial right semi-model category $\mathbf{M}$ exists and it is combinatorial. Moreover, if $\mathbf{M}$ is simplicial, then so is the localized weak model structure.
\end{cor}

\begin{proof}
Let $\mathbf{K} = \mathcal{S}^{\top}$. By \Cref{cellularizationvopenka}, the cellularization of $\mathbf{M}$ at $\mathbf{K}$ exists and it is combinatorial. We also know that it is computed as the right Bousfield localization at $\mathbb{S}_{\mathbf{K}} = \mathbf{K}^{\top}$. It remains to show that it corresponds with the right Bousfield localization at $\mathcal{S}$. Let us compute

\begin{center}
$\mathbb{S}_{\mathbf{K}}^{\top} = (\mathbf{K}^{\top})^{\top} = ((\mathcal{S}^{\top})^{\top})^{\top} = \mathcal{S}^{\top}$.
\end{center}

In particular, $\mathbb{S}_{\mathbf{K}}$-colocal objects correspond to $\mathcal{S}$-colocal objects, so that in turn $\mathbb{S}_{\mathbf{K}}$-colocal equivalences correspond to $\mathcal{S}$-colocal equivalences, which concludes the proof.
\end{proof}

\begin{prop} \label{infcolocalization}
Assume $\vp$, and let $\mathcal{C}$ be a presentable $\infty$-category. Then for every class of maps $S$ in $\mathcal{C}$, the colocalization of $\mathcal{C}$ at $S$ is a coreflective subcategory, and it is itself presentable.
\end{prop}

\begin{rem}
The following proof requires $\vp$ even when $S$ is a set.
\end{rem}

\begin{proof}
By \cite{HTT}, Proposition A.3.7.6, there is a combinatorial model category $\mathbf{M}$ such that $\N_{\Delta}(\mathbf{M}^{\circ}) \simeq \mathcal{C}$. Let $\mathbb{S}$ be the class of morphisms in $\mathbf{M}^{\circ}$ that give rise to morphisms in $S$. By \Cref{rightbousfieldvopenka}, the right Bousfield localization $R_{\mathbb{S}} \mathbf{M}$ exists and it is combinatorial. Moreover, since there is a Quillen adjunction

\begin{center}
$\mathbf{M} \leftrightarrows R_{\mathbb{S}} \mathbf{M}$
\end{center}

where the two functors are both identities on the underlying categories, using \Cref{adjunctionweakmodel} we obtain an adjunction of the underlying $\infty$-categories

\begin{center}
$\mathcal{C} \leftrightarrows (R_{\mathbb{S}} \mathbf{M})^{\infty}$
\end{center}

where the left adjoint is easily seen to be fully faithful after an inspection of mapping spaces. \Cref{combpres} says that $(R_{\mathbb{S}} \mathbf{M})^{\infty}$ is presentable. It only remains to show that it is equivalent to the colocalization at $S$.\\
For this, we are reduced to prove that an object of $\mathbf{M}$ is $\mathbb{S}$-colocal if and only if it is $S$-colocal as an object of $\mathcal{C}$. Since homotopy co-orthogonality is detected on underlying $\infty$-categories, this is true by definition.
\end{proof}

\section{Open questions}
To conclude this treatment, we briefly list some of the questions that are strictly related to what we have been discussing but that have to this point remained unanswered.

\begin{enumerate}
\item Is $\vp$ equivalent to $\ivp$? As pointed out in \Cref{vpandivp}, there is reason to believe that this is the case, but as of now we still lack a mathematical proof. The interest of this lies in the fact that, if true, it would add one more item to the list of manifestations of akinness between local presentability in the 1-dimensional world and in the $\infty$-dimensional one, constituting one more instance of the former being  enough to also cover at least a significant part of the theory of the latter.

\item[(1)$'$] Is $\ivp$ equivalent to a formulation given in terms of rigid classes in presentable $\infty$-categories? In other words, is there an $\infty$-categorical version of \Cref{vopenka}$(1) \Leftrightarrow (2)$? Of course, since we have the mentioned proposition and \Cref{altvopenkaivopenka}, an affirmative answer to this is equivalent to an affirmative answer to question (1). In the proof of \Cref{vopenka}, a central role is played by the existence of a full embedding $\mathcal{K} \hookrightarrow \catname{Gra}$ of a presentable category into the category of graphs, which in particular lands in the subcategory of graphs for which the map $(s,t): E \to V \times V$ is monic. This latter property is used to rest assured that there is a bounded amount of such graphs with a fixed set of vertices. The existence of such embeddings for $\infty$-categories is proven in \Cref{univembedding}, but unfortunately we don't know that the monic property is verified. If this were true, it would allow an adaptation of \Cref{vopenka} to $\infty$-categories, and hence an affirmative answer to (1) and (1)$'$.

\item Is there any compact way to characterize co-orthogonality classes in presentable ($\infty$-)categories that can be expressed as small co-orthogonality classes under $\vp$? The main reason for asking this is that many of the attempted and failed strategies aimed to find such a characterization fell rather short of proving that \emph{all} co-orthogonality classes are small, with only relatively minor details here and there rowing against it, which makes it a natural question to ask if there are any minimal conditions that would make such a proof actually work.

\item Can the implication $(1) \Rightarrow (3)$ of \Cref{colocalizationcoreflective} be generalized to $\infty$-categories? The original proof in \cite{LocPresAccCat}, Corollary 6.29 seems to defy an $\infty$-categorical adaptation, in that it depends on the special adjoint functor theorem, which in turns seems unwilling to be generalized, due to the much harder recognition of monomorphisms in $\infty$-categories as compared to 1-categories.

\item Weak Vop\v{e}nka's principle can be made considerably looser by replacing a fully faithful functor $\Ord^{op} \to \mathcal{K}$ with an assignment $\Ord^{op} \to \mathcal{K}$ which is surjective but not necessarily injective on morphisms, and not even necessarily functorial. This is precisely the statement of $\swvp$. Is there a similar rephrasing of Vop\v{e}nka's principle? Its existence would have very nice consequences, such as almost immediately positive answers to questions (1) and (3), the latter because it would give a way to obviate the need of SAFT.

\item As already mentioned in \Cref{smalllimclosureremark}, we know without set-theoretical assumptions that the limit closure of a small full subcategory of a presentable category is reflective. It remains open whether this is also true for $\infty$-categories, since our only proof so far relies on $\wvp$.
\end{enumerate}

\appendix
\section{Weak model and semi-model categories}
This appendix is devoted to working out some results that are needed towards the end of \Cref{sectioncellularizations}. I would like to thank Simon Henry who pointed me to his own articles and clarified several subtle differences between Quillen model categories and weak model categories.\\
The full theory of weak model categories is not relevant to our scopes, but a treatment is given in \cite{Weakmodel}, and further in \cite{Combinatorialweakmodel} for the accessible and combinatorial case. The idea of a weak model category is requiring the same axioms as for a Quillen model category, except that only maps with cofibrant domain and fibrant codomain are required to factor, and fibrations (cofibrations) are only required to lift against acyclic cofibrations between cofibrant objects (acyclic fibrations between fibrant objects). In particular, there are in general no replacements for all objects, but all fibrant objects have a cofibrant replacement and all cofibrant objects have a fibrant replacement, which can be chosen to be functorial in the combinatorial setting.\\
In practice, we will actually need the stronger notion of semi-model category, coming in the two dual forms of left and right semi-model categories and sitting in between Quillen and weak model categories. A treatment is given in \cite{Barwickleftright}, which we will refer to for some results that we need. For the sake of easier track-keeping, we will spell out the main definitions concerning right semi-model category.

\begin{defn}
A right semi-model category consists of a (finitely, depending on taste) complete and cocomplete category $\mathcal{C}$ and three classes of morphisms, called weak equivalences, fibrations and cofibrations, such that:

\begin{enumerate}
\item All three classes are closed under composition and retracts;
\item the class of weak equivalences satisfies the 2-out-of-3 property;
\item the final object is fibrant;
\item the classes of cofibrations and trivial fibrations with fibrant codomain form a weak factorization system;
\item the classes of trivial cofibrations and fibrations form a weak factorization system.
\end{enumerate}
\end{defn}

\begin{defn}
An adjunction between right semi-model categories is called a Quillen adjunction if the left adjoint preserves cofibrations and trivial cofibrations.
\end{defn}

\begin{defn}
A right semi-model category is called combinatorial if its underlying category is presentable and there are two small sets of morphisms $I$ and $J$ such that:

\begin{itemize}
\item a morphism with fibrant codomain is a trivial fibration if and only if it has the right lifting property against $I$;
\item a morphism is a fibration if and only if it satisfies the right lifting property against $J$.
\end{itemize}
\end{defn}

\begin{rem} \label{beforedklocalizations}
Unlike for Quillen model categories, the definition of Quillen adjunction is not equivalent to the right adjoint preserving fibrations and trivial fibrations. Dually, a Quillen adjunction between \emph{left} semi-model categories has to be defined as the right adjoint preserving those two classes.\\
Furthermore, the definition of combinatorial includes a version of cofibrant generation. In general, if we are dealing with a Quillen model category, its cofibrant generation as a right semi-model category does \emph{not} imply cofibrant generation as a model category.
\end{rem}

In the following, we will freely refer to the underlying $\infty$-categories of relative categories. Depending on the specific classes of relative categories and flavors of model categories that one is handling, there may be various equivalent ways of constructing such $\infty$-categories. However, our discussion is not bound to any specific one of them, so we will only use the general properties that all of them must satisfy. The most general way of constructing the underlying $\infty$-category is the following: given a relative category $(\mathcal{C}, \mathcal{W})$, take the pushout

\begin{center}
\begin{tikzcd}
\N (\mathcal{W}) \ar[d] \ar[r] & \N (\mathcal{C}) \ar[d]\\
\textup{Ex}^{\infty} \N (\mathcal{W}) \ar[r] & \textup{Loc}'(\mathcal{C},\mathcal{W})
\end{tikzcd}
\end{center}

in $\sSet$, then $\textup{Loc}(\mathcal{C},\mathcal{W})$ will be a Joyal-fibrant replacement of $\textup{Loc}' (\mathcal{C},\mathcal{W})$.\\
It will be useful to know that, whenever $\underline{\mathbf{M}}$ is a simplicial model category, with $\mathcal{W}$ the class of weak equivalences, and $\mathbf{M}$ is its underlying ordinary category, then there is an equivalence of $\infty$-categories

\begin{center}
$\N_{\Delta}(\underline{\mathbf{M}}^{bif}) \simeq \textup{Loc}(\mathbf{M},\mathcal{W})$
\end{center}

which provides another model of the underlying $\infty$-category for this specific class of categories. This result was first formulated in \cite{SimpLoc3}, Proposition 4.8. An alternative approach is taken in \cite{SimpLocrevisited}, Corollary 1.4.3, and yet another proof, taking place directly in $\sSet$ instead of going through $\sCat$, appears in \cite{SimpLocCisinski}.\\
When not referring to a specific model for the underlying $\infty$-category of a relative, or model category $\mathcal{C}$, we will simply denote it by $\mathcal{C}^{\infty}$. Moreover, a functor $F$ between such categories induces a functor $F^{\infty}$ between underlying $\infty$-categories.\\

Before stating the main results of the section, we need some preparation.

\begin{lem} \label{adj1}
Given two morphisms $u: X \to Y$ and $u': X' \to Y'$ in an $\infty$-category $\mathcal{C}$, then there is a homotopy commutative square

\begin{center}
\begin{tikzcd}
\mathcal{C}^{\Delta^1}(u,u') \ar[r] \ar[d] & \mathcal{C}(Y,Y') \ar[d, "u^{\ast}"]\\
\mathcal{C}(X,X') \ar[r, "u'_{\ast}", swap] & \mathcal{C}(X,Y')
\end{tikzcd}
\end{center}

where the unmarked maps are induced by the two projections $\mathcal{C}^{\Delta^1} \to \mathcal{C}$.
\end{lem}

\begin{proof}
Let us consider the two commutative diagrams, living one on top of the other



\begin{center}
\begin{tikzcd}
&& \mathcal{C}^{\Delta^1 \times \Delta^1} \ar[dll] \ar[dl] \ar[dr] \ar[drr]\\
\mathcal{C}^{0 \times \Delta^1} & \mathcal{C}^{\Delta^2} \ar[l, "01"] \ar[dr, "02", swap] && \mathcal{C}^{\Delta^2} \ar[r, "12", swap] \ar[dl, "02"] & \mathcal{C}^{1 \times \Delta^1}\\
&& \mathcal{C}^{\Delta^1} \ar[d, Rightarrow, shorten <= 5pt, shorten >= 5pt]\\
&& \mathcal{C}^{\Delta^1} \times \mathcal{C}^{\Delta^1} \ar[dll, "p_0 \times p_0", swap] \ar[dl, "p_0 \times \id"] \ar[dr, "\id \times p_1", swap] \ar[drr, "p_1 \times p_1"]\\
\mathcal{C} \times \mathcal{C} & \mathcal{C} \times \mathcal{C}^{\Delta^1} \ar[l, "\id \times p_0"] \ar[dr, "\id \times p_1", swap] && \mathcal{C}^{\Delta^1} \times \mathcal{C} \ar[r, "p_1 \times \id", swap] \ar[dl, "p_0 \times \id"] & \mathcal{C} \times \mathcal{C}\\
&& \mathcal{C} \times \mathcal{C}
\end{tikzcd}
\end{center}

where: the two maps $\mathcal{C}^{\Delta^1 \times \Delta^1} \to \mathcal{C}^{\Delta^2}$ restrict to the upper and lower triangle respectively, so that the composite $\mathcal{C}^{\Delta^1 \times \Delta^1} \to \mathcal{C}^{\Delta^1}$ selects the diagonal; the map $\mathcal{C}^{\Delta^1 \times \Delta^1} \to \mathcal{C}^{\Delta^1} \times \mathcal{C}^{\Delta^1}$ is induced by the inclusion $\Delta^1 \times 0 \cup \Delta^1 \times 1 \hookrightarrow \Delta^1 \times \Delta^1$; the map $\mathcal{C}^{\Delta^2} \to \mathcal{C} \times \mathcal{C}^{\Delta^1}$ is induced by the inclusion $0 \times 0 \cup 1 \times \Delta^1 \hookrightarrow \Delta^2$; the map $\mathcal{C}^{\Delta^2} \to \mathcal{C}^{\Delta^1} \times \mathcal{C}$ is induced by the inclusion $\Delta^1 \times 0 \cup 1 \times 1 \hookrightarrow \Delta^2$; all three maps $\mathcal{C}^{\Delta^1} \to \mathcal{C} \times \mathcal{C}$ are induced by the inclusion $0 \cup 1 \hookrightarrow \Delta^1$.\\
Now consider the point $(u,u') \in \mathcal{C}^{\Delta^1} \times \mathcal{C}^{\Delta^1}$, and observe that it is transported to the following points in the base diagram

\begin{center}
\begin{tikzcd}
&& (u,u') \ar[dll, mapsto] \ar[dl, mapsto] \ar[dr, mapsto] \ar[drr, mapsto]\\
(X,X') & (X,u') \ar[l, mapsto] \ar[dr, mapsto] && (u,Y') \ar[r, mapsto] \ar[dl, mapsto] & (Y,Y')\\
&& (X,Y').
\end{tikzcd}
\end{center}

Therefore, we have an induced diagram of fibers, which, by a repeated use of \cite{HTT}, Corollary 4.2.1.8, looks like this

\begin{center}
\begin{tikzcd}
&& \mathcal{C}^{\Delta^1}(u,u') \ar[dll] \ar[dl] \ar[dr] \ar[drr]\\
\mathcal{C}(X,X') & \mathcal{C}_{/u'} \times_{\mathcal{C}} \{ X \} \ar[l] \ar[dr] && \mathcal{C}_{u/} \times_{\mathcal{C}} \{ Y' \} \ar[dl] \ar[r] & \mathcal{C}(Y,Y')\\
&& \mathcal{C}(X,Y').
\end{tikzcd}
\end{center}

Now, by \cite{HTT}, Proposition 2.1.2.5, the projection $\mathcal{C}_{u/} \to \mathcal{C}_{Y/}$ is a trivial Kan fibration, therefore its fiber on the point $\{ Y' \}$, which is precisely the right horizontal map in the diagram, is also a trivial Kan fibration, and in particular an equivalence. Dually, the left horizontal map is an equivalence. The central square represents thus the square in the statement, and the proof is complete.
\end{proof}

\begin{lem} \label{adj1bis}
Let $f: X' \to X$ and $g: Y \to Y'$ be two morphisms in an $\infty$-category $\mathcal{C}$. Then there is a homotopy commutative square

\begin{center}
\begin{tikzcd}
\mathcal{C}(X,Y) \ar[d, "f^{\ast}"] \ar[r, "g_{\ast}"] & \mathcal{C}(X,Y') \ar[d, "f^{\ast}"]\\
\mathcal{C}(X',Y) \ar[r, "g_{\ast}"] & \mathcal{C}(X',Y').
\end{tikzcd}
\end{center}
\end{lem}

\begin{proof}
Consider the following commutative diagram

\begin{center}
\begin{tikzcd}
&& \mathcal{C}^{\Delta^3} \ar[dll, "02", swap] \ar[dl, "023"] \ar[dr, "013", swap] \ar[drr, "13"]\\
\mathcal{C}^{\Delta^1} & \mathcal{C}^{\Delta^2} \ar[l, "01"] \ar[dr, "02", swap] && \mathcal{C}^{\Delta^2} \ar[r, "12", swap] \ar[dl, "02"] & \mathcal{C}^{\Delta^1}\\
&& \mathcal{C}^{\Delta^1}
\end{tikzcd}
\end{center}

so that the vertical composite $\mathcal{C}^{\Delta^3} \to \mathcal{C}^{\Delta^1}$ is $03$. Moreover, this diagram lives over a similarly shaped diagram that looks like the one in the proof of \Cref{adj1}, where the map $\mathcal{C}^{\Delta^3} \to \mathcal{C}^{\Delta^1} \times \mathcal{C}^{\Delta^1}$ is induced by the inclusion $01 \cup 23: \Delta^1 \cup \Delta^1 \hookrightarrow \Delta^3$ and all other maps are the same. Analogously as in the proof of \Cref{adj1}, chasing the point $(f,g) \in \mathcal{C}^{\Delta^1} \times \mathcal{C}^{\Delta^1}$ gives a homotopy commutative square

\begin{center}
\begin{tikzcd}
\mathcal{C}^{\Delta^3} \times_{\mathcal{C}^{\Delta^1} \times \mathcal{C}^{\Delta^1}} \{ (f,g) \} \ar[r] \ar[d] & \mathcal{C}(X,Y') \ar[d]\\
\mathcal{C}(X',Y) \ar[r] & \mathcal{C}(X',Y').
\end{tikzcd}
\end{center}

We conclude by observing that, once more by a similar argument as in the final passage in the proof of \Cref{adj1}, there is an equivalence $\mathcal{C}^{\Delta^3} \times_{\mathcal{C}^{\Delta^1} \times \mathcal{C}^{\Delta^1}} \{ (f,g) \} \simeq \mathcal{C}(X,Y)$ which is compatible with the desired square.
\end{proof}

\begin{lem} \label{adj2}
Given a natural transformation $\alpha: f \Rightarrow g: \mathcal{C} \to \mathcal{D}$ between functors of $\infty$-categories, then for every objects $X,Y \in \mathcal{C}$ there is a homotopy commutative square

\begin{center}
\begin{tikzcd}
\mathcal{C}(X,Y) \ar[r, "f"] \ar[d, "g"] & \mathcal{D}(fX,fY) \ar[d, "\alpha_{Y \ast}"]\\
\mathcal{C}(gX,gY) \ar[r, "\alpha_X^{\ast}"] & \mathcal{C}(fX,gY).
\end{tikzcd}
\end{center}
\end{lem}

\begin{proof}
\Cref{adj1} provides a homotopy commutative square

\begin{center}
\begin{tikzcd}
\mathcal{D}^{\Delta^1}(\alpha_X, \alpha_Y) \ar[r] \ar[d] & \mathcal{D}(fX,fY) \ar[d]\\
\mathcal{D}(gX,gY) \ar[r] & \mathcal{D}(fX,gY)
\end{tikzcd}
\end{center}

so it suffices to show that there is a map $\mathcal{C}(X,Y) \to \mathcal{D}^{\Delta^1}(\alpha_X,\alpha_Y)$ whose composition with the relevant maps in the square above gives $f$ and $g$ respectively. This is true simply because $\alpha$ can be viewed as a functor $\mathcal{C} \to \mathcal{D}^{\Delta^1}$ satisfying the desired compatibilities.
\end{proof}

\begin{lem} \label{adj3}
Given two functors $l: \mathcal{C} \leftrightarrows \mathcal{D}: r$ between $\infty$-categories, and a natural transformation $\varepsilon: lr \Rightarrow 1_{\mathcal{D}}$, there are homotopy commutative triangles

\begin{center}
\begin{tikzcd}
\mathcal{D}(lC,Y) \ar[d, "lr"] \ar[dr, bend left, "\varepsilon_{lC}^{\ast}"]\\
\mathcal{D}(lrlC,lrD) \ar[r, "\varepsilon_{D \ast}"] & \mathcal{D}(lrlC,Y)
\end{tikzcd}
\end{center}

for every two objects $C \in \mathcal{C}$ and $D \in \mathcal{D}$.
\end{lem}

\begin{proof}
This is just a special case of \Cref{adj2}, in the case $\alpha = \varepsilon$.
\end{proof}

\begin{prop} \label{triangleidentities}
Given two functors $l: \mathcal{C} \leftrightarrows \mathcal{D}: r$ between $\infty$-categories, and natural transformations $\eta: 1_{\mathcal{C}} \Rightarrow rl$ and $\varepsilon: lr \Rightarrow 1_{\mathcal{D}}$, then these data constitute an adjunction if and only if the homotopy triangle identities

\begin{center}
\begin{tikzcd}
& rlr \ar[dr, "r \varepsilon"]\\
r \ar[ur, "\eta r"] \ar[rr, "\id_r"] && r
\end{tikzcd}
\hspace{30pt}
\begin{tikzcd}
& lrl \ar[dr, "\varepsilon l"]\\
l \ar[ur, "l \eta"] \ar[rr, "\id_l"] && l
\end{tikzcd}
\end{center}

are satisfied.
\end{prop}

\begin{proof}
Suppose we already have an adjunction. Then by \cite{HTT}, Proposition 5.2.2.12 this induces an adjunction on homotopy categories, for which the triangle identities are satisfied strictly by the classical theorem, but this means that they are satisfied up to homotopy in $\mathcal{C}$ and $\mathcal{D}$.\\
Conversely, assume that we have the triangle identities. We want to use \cite{HTT}, Proposition 5.2.2.8, so it will suffice to show that for every $C \in \mathcal{C}$ and $D \in \mathcal{D}$ the map $\eta_C^{\ast} \circ r: \mathcal{D}(lC,D) \to \mathcal{C}(rlC,rD) \to \mathcal{C}(C,rD)$ is a homotopy equivalence. It will suffice to show that it has a two-sided homotopy inverse. We claim that this inverse is $\varepsilon_{D \ast} \circ l$. We will only verify that it is a left inverse, the fact that it is a right inverse is proven dually. Now let us compute

\begin{alignat*}{3}
& \varepsilon_{D \ast} \circ l \circ \eta_C^{\ast} \circ r & \simeq \varepsilon_{D \ast} \circ (l \eta_C)^{\ast} \circ lr & \quad \text{by functoriality of $l$}\\
&& \simeq (l \eta_C)^{\ast} \circ \varepsilon_{D \ast} \circ lr & \quad \text{by \Cref{adj1bis}}\\
&& \simeq (l \eta_C)^{\ast} \circ \varepsilon_{lC}^{\ast} & \quad \text{by \Cref{adj3}}\\
&& \simeq \id & \quad \text{by a triangle identity}.
\end{alignat*}
\end{proof}

The next theorems will only be used in the case of semi-model categories, but since their proof apply step by step to weak model categories, we will formulate them in the more general case.

\begin{prop} \label{adjunctionweakmodel}
Let $L: \mathcal{C} \leftrightarrows \mathcal{D}: R$ be a Quillen adjunction between weak model categories. Then the derived functors $L^{\infty}: \mathcal{C}^{\infty} \leftrightarrows \mathcal{D}^{\infty}: R^{\infty}$ between the underlying $\infty$-categories form an adjunction.
\end{prop}

\begin{proof}
The general case can be proven in essentially the same way as in \cite{MazelGeeadjunctions}, which works for weak model categories with no changes. However, since the cases we will be interested in all come with functorial factorizations, we will give here an explicit proof with this assumption.\\
Let us call $\eta$ and $\varepsilon$ the unit and counit of the adjunction, and let us also call $\mu: Q \to 1$ and $\rho: 1 \to P$ a cofibrant and a fibrant replacement that preserve fibrancy and cofibrancy respectively. Observe that the functor $L$ does not generally preserve weak equivalences. However, by Ken Brown's lemma, the restricted functor $L: \mathcal{C}^{cof} \to \mathcal{D}^{cof}$ does. The problem is that the analogous observation for $R$ yields a homotopical functor $R: \mathcal{D}^{fib} \to \mathcal{C}^{fib}$, creating a mismatch between the domains and codomains of the two functors. However, the inclusions of the full subcategories of bifibrant objects into those of fibrant (cofibrant) objects have a cofibrant (fibrant) replacement as homotopy inverse, so there are induced functors $PL: \mathcal{C}^{bif} \leftrightarrows \mathcal{D}^{bif}: QR$ that are equivalent to $L$ and $R$. We will define a derived unit and counit for these, and prove that they satisfy the triangular identities up to homotopy, which is enough by \Cref{triangleidentities}.\\
For an object $C \in \mathcal{C}^{bif}$, consider the morphisms

\begin{center}
\begin{tikzcd}
C & QC \ar[l, "\mu_C", swap] \ar[r, "Q \eta_C"] & QRLC \ar[r, "QR \rho_{LC}"] & QRPLC.
\end{tikzcd}
\end{center}

Both $C$ and $QRPLC$ are bifibrant and, moreover, $\mu_C$ becomes invertible when passing to the underlying $\infty$-categories, so that this induces a morphism $C \to QRPLC$ that is natural in $C$ in an $\infty$-categorical sense. Analogously, if $D \in \mathcal{D}^{bif}$, the morphisms

\begin{center}
\begin{tikzcd}
PLQRD \ar[r, "PL \mu_{RD}"] & PLRD \ar[r, "P \varepsilon_D"] & PD & D \ar[l, "\rho_D", swap]
\end{tikzcd}
\end{center}

induce a natural morphism $PLQRD \to D$.\\
Now consider the diagram, for an object $D \in \mathcal{D}^{bif}$:

\begin{center}
\begin{tikzcd}[column sep = huge, row sep = huge]
QRD \ar[dr, phantom, "1"] \ar[d, "\mu_{RD}"] & QQRD \ar[dr, phantom, "2"] \ar[l, "\mu_{QRD}", swap] \ar[d, "Q \mu_{RD}"] \ar[r, "Q \eta_{QRD}"] & QRLQRD \ar[dr, phantom, "3"] \ar[d, "QRL \mu_{RD}"] \ar[r, "QR \rho_{LQRD}"] & QRPLQRD \ar[d, "QRPL \mu_{RD}"]\\
RD & QRD \ar[l, "\mu_{RD}", swap] \ar[r, "Q \eta_{RD}"] \ar[rd, "Q \id_{RD}", swap] & QRLRD \ar[dr, phantom, "4"] \ar[d, "QR \varepsilon_D"] \ar[r, "QR \rho_{LRD}"] & QRPLRD \ar[d, "QRP \varepsilon_D"]\\
&& QRD \ar[r, "QR \rho_D"] & QRPD.
\end{tikzcd}
\end{center}

The square 1 commutes by naturality of $\mu$, the square 2 by naturality of $\eta$ and functoriality of $Q$, the square 3 by naturality of $\rho$ and functoriality of $QR$, the square 4 by naturality of $\rho$ and functoriality of $QR$. Moreover, the triangle is given by applying $Q$ to one of the strict triangular identities. Now, the subdiagram comprising 2, 3, 4 and the triangle can be rewritten as

\begin{center}
\begin{tikzcd}[column sep = huge]
QRLQRD \ar[r, "QR \rho_{LQRD}"] & QRPLQRD \ar[r, "QRPL \mu_{RD}"] & QRPLRD \ar[d, "QRP \varepsilon_D"]\\
QQRD \ar[u, "Q \eta_{QRD}"] \ar[d, "Q \mu_{RD}"] && QRPD\\
QRD \ar[rr, "\id_{QRD}"] && QRD \ar[u, "QR \rho_D"].
\end{tikzcd}
\end{center}

By the square 1 in the previous diagram and the fact that $\mu_{RD}$ becomes homotopy invertible in the underlying $\infty$-category, the vertical map $Q \mu_{RD}$ is homotopic to $\mu_{QRD}$, so we can replace it in the latter diagram and still maintain commutativity up to homotopy. Now, the lower vertical maps are both homotopy invertible in the underlying $\infty$-category, so that the resulting diagram precisely expresses one of the homotopy triangular identities. The other one is proven dually.
\end{proof}

\begin{cor} \label{equivalenceweakmodel}
If two weak model categories are Quillen equivalent, then their underlying $\infty$-categories are equivalent.
\end{cor}

\begin{proof}
Since there is a Quillen adjunction between them, \Cref{adjunctionweakmodel} shows that there is an $\infty$-categorical adjunction. It suffices to show that the derived unit and counit are natural equivalences. For the unit, observe that $\mu_C$ is always a weak equivalence, so it remains to show that the composite

\begin{center}
\begin{tikzcd}
QC \ar[r, "Q \eta_C"] & QRLC \ar[r, "QR \rho_{LC}"] & QRPLC
\end{tikzcd}
\end{center}

is a weak equivalence. Since $Q$ preserves all weak equivalences, it suffices to see that $R \rho_{LC} \circ \eta_C: C \to RPLC$ is a weak equivalence. But this is the transpose map of $\rho_{LC}: LC \to PLC$ under the adjunction $L \dashv R$. Since this is a Quillen equivalence, it suffices to observe that $C$ is cofibrant, $PLC$ is fibrant, and $\rho_{LC}$ is always a weak equivalence. A dual argument shows that the derived counit is a natural equivalence.
\end{proof}

\begin{thm} \label{combpres}
The underlying $\infty$-category of a combinatorial weak model category is presentable.
\end{thm}

\begin{proof}
By \cite{Algcoffib}, Theorem 35, a combinatorial weak model category is connected by a chain of Quillen equivalences to a combinatorial Quillen model category, which in turn is equivalent to a combinatorial simplicial Quillen model category by \cite{DuggerCombinatorialPresentation}. Therefore, in view of \Cref{equivalenceweakmodel}, the underlying $\infty$-category of a combinatorial weak model category is equivalent to that of a combinatorial simplicial Quillen model category, which is presentable by \cite{HTT}, Proposition A.3.7.6 (using the alternative description of the underlying $\infty$-category of a simplicial model category).
\end{proof}

\section{Universal embedding for $\infty$-categories} \label{univembedding}
The purpose of this appendix is to partly generalize to the $\infty$-world a result that is known for accessible categories: the existence of a single category that contains all of them as full subcategories. In fact, there is a number of such categories, many of which are expounded for example in \cite{CombAlgTopRep}. The one we will be generalizing is the embedding into the category of graphs, defined as sets with a binary relation. This is a rather non-trivial result that I learned in \cite{LocPresAccCat}, Theorem 2.65. The embedding there constructed also has the property of being accessible, but whether we can obtain this property in the $\infty$-categorical case remains unknown as of yet.\\
Let $G$ be the category containing two distinct points and two distinct parallel arrows between them, i.e. the diagram of shape

\begin{center}
$\bullet \rightrightarrows \bullet$.
\end{center}

For the purposes of book-keeping, we recall the 1-categorical result.

\begin{thm}[\cite{LocPresAccCat}, 2.65] \label{2.65}
Let $\mathcal{K}$ be an accessible category. Then there is an accessible full embedding $\mathcal{K} \hookrightarrow \Set^G$.
\end{thm}

We can now set out to prove an $\infty$-categorical version of it, which we now state, before starting to pave the way that leads to its proof. Remember that, as usual, we denote by $\mathcal{S}$ the $\infty$-category of spaces.

\begin{thm} \label{i2.65}
Let $\mathcal{K}$ be an accessible $\infty$-category. Then there is a full embeddding $\mathcal{K} \hookrightarrow \mathcal{S}^G$.
\end{thm}

\begin{lem} \label{accessibletopresheaves}
Let $\mathcal{K}$ an accessible $\infty$-category. Then there is a full embedding of the form $\mathcal{K} \hookrightarrow \mathcal{P(C)}$, where $\mathcal{C}$ is a small $\infty$-category.
\end{lem}

\begin{proof}
This is clear from the definition of accessibility, i.e. \cite{HTT} Definitions 5.4.2.1 and 5.3.5.1.
\end{proof}

\begin{lem} \label{combinatorialtopresheaves}
Given a combinatorial model category $\mathbf{M}$, there is a small ordinary category $\mathcal{C}$ such that $\mathbf{M}$ is a left Bousfield localization of $\sSet^{\mathcal{C}}$. Moreover, this induces a fully faithful functor $\mathbf{M}^{\infty} \hookrightarrow (\sSet^{\mathcal{C}})^{\infty}$ between the underlying $\infty$-categories.
\end{lem}

\begin{proof}
This is \cite{DuggerCombinatorialPresentation}, Theorem 1.1. The category $R \mathcal{C}$ used in its proof in the unenriched case is an ordinary category, so it is precisely what we are here renaming $\mathcal{C}$.\\
To see that this induces a fully faithful functor, observe first that, again referring to \cite{DuggerCombinatorialPresentation}, Theorem 1.1, there is a Quillen equivalence between $\mathbf{M}$ and a left Bousfield localization of $\sSet^{\mathcal{C}}$, therefore by \Cref{equivalenceweakmodel} there is an equivalence of $\infty$-categories between $\mathbf{M}^{\infty}$ and such a localization $(L_S \sSet^{\mathcal{C}})^{\infty}$. It therefore remains to show that the right adjoint identity functor $L_S \sSet^{\mathcal{C}} \to \sSet^{\mathcal{C}}$ induces a fully faithful functor on underlying $\infty$-categories.\\
To this end, observe that by \Cref{hirschhornleft} the Bousfield localization at issue is simplicial, meaning in particular that the functors of which it is comprised are themselves simplicial. Since, by the discussion immediately following \Cref{beforedklocalizations} the underlying $\infty$-category of a simplicial model category can be computed as $\N_{\Delta}((-)^{\circ})$, mapping spaces and actions of functors are detected at the level of simplicial categories, and the action of the functor at issue is given by the identity, so it obviously induces a fully faithful derived functor.
\end{proof}

\begin{rem}
In the following, it will be of key importance, rather unusually, that $\mathcal{C}$ is \emph{not} simplicial, in that this will allow to juggle around $\Delta$ just like any other category that we will be considering, forgetting that it constitutes the basis of all our enrichments.
\end{rem}

Now we need some constructions. Let us fix a small ordinary category $\mathcal{C}$, and let us look at the category $\sSet^{\mathcal{C}}$. This can be equivalent regarded as $\Set^{\mathcal{C} \times \Delta^{op}}$. Now, the category $\Set^{\mathcal{C}}$ is presentable, therefore \Cref{2.65} provides a full embedding $\gamma: \Set^{\mathcal{C}} \hookrightarrow \Set^G$. Our next step is defining an embedding $\Set^{\mathcal{C} \times \Delta^{op}} \to \Set^{G \times \Delta^{op}}$.\\

\begin{defn}
Let $F: \mathcal{C} \times \Delta^{op} \to \Set$ be a functor. It can be regarded as a sequence of functors $(F_n: \mathcal{C} \to \Set)_{n \in \Delta^{op}}$ with appropriate natural transformations between them.\\
We want to define a functor $\delta F: G \times \Delta^{op} \to \Set$, i.e. a sequence of functors $G \to \Set$ and natural transformations between them indexed by $\Delta^{op}$. For each $n$, we take the functor $\gamma F_n$. Since $\gamma$ is a functor, this assignment is functorial on $F$. Therefore, we have defined a functor

\begin{center}
$\delta: \Set^{\mathcal{C} \times \Delta^{op}} \to \Set^{G \times \Delta^{op}}$.
\end{center}
\end{defn}

\begin{prop} \label{deltaunenrichedfff}
The functor $\delta$ is fully faithful.
\end{prop}

\begin{proof}
Recalling the expression of natural transformations as ends (for example, see \cite{CWM}, IX.5), compute for two functors $F$ and $G$

\begin{alignat*}{3}
& \Set^{\mathcal{C} \times \Delta^{op}}(F,G) & \cong \int_{n \in \Delta^{op}} \Set^{\mathcal{C}}(F_n,G_n)\\
&& \cong \int_{n \in \Delta^{op}} \Set^G(\gamma F_n, \gamma G_n) & \quad \textup{because $\gamma$ is fully faithful}\\
&& \cong \Set^{G \times \Delta^{op}} (\delta F, \delta G) & \quad \textup{by definition of $\delta$}.
\end{alignat*}
\end{proof}

Observe that $\Set^{\mathcal{C} \times \Delta^{op}}$ has the structure of a simplicial category given as follows: let $\Delta^n$ denote the functor $\mathcal{C} \times \Delta^{op} \to \Set$ defined by $(C,m) \mapsto [m,n]$, i.e. forgets $C$ and acts as the usual representable simplicial set $\Delta^n$. For two functors $F$ and $G$ now define

\begin{center}
$\Set^{\mathcal{C} \times \Delta^{op}} (F,G)_n = \Set^{\mathcal{C} \times \Delta^{op}} (F \times \Delta^n, G)$.
\end{center}

Then the sequence $\Set^{\mathcal{C} \times \Delta^{op}}(F,G)_{\bullet}$ has the structure of a simplicial set, and gives a simplicial enrichment. Similarly, $\Set^{G \times \Delta^{op}}$ is simplicially enriched. Our next goal will be to extend $\delta$ to a simplicial functor.\\
This is made a little difficult by the fact that requiring that $\gamma$ preserve products is very unreasonable. However, it \emph{almost} preserves products with discrete objects, as will be made precise in a second. Given a functor $F: \mathcal{C} \to \Delta^{op} \to \Set$, let us regard it as a functor $\Delta^{op} \to \Set^{\mathcal{C}}$. The functor $\delta F \times \Delta^n$ is defined by the assignation

\begin{center}
$m \mapsto \gamma F_m \times [m,n] \cong \coprod_{m \to n} \gamma F_m$.
\end{center}

Analogously, the functor $\delta (F \times \Delta^n)$ is defined by the assignation

\begin{center}
$m \mapsto \gamma (F_m \times [m,n]) \cong \gamma (\coprod_{m \to n} F_m)$.
\end{center}

\begin{defn}
The universal property of coproducts yields a natural map

\begin{center}
$\coprod_{m \to n} \gamma F_m \to \gamma (\coprod_{m \to n} F_m)$.
\end{center}

This induces a natural transformation $\theta: \delta F \times \Delta^n \to \delta (F \times \Delta^n)$.
\end{defn}

In general, $\theta$ is not an isomorphism, but the next lemma says that it behaves like one when we narrow down our attention to the image of $\delta$.

\begin{lem} \label{keylemma}
For two functors $F,G: \mathcal{C} \times \Delta^{op} \to \Set$ and $n \in \Delta^{op}$, the induced map

\begin{center}
$\theta^{\ast}: \Set^{G \times \Delta^{op}}(\delta(F \times \Delta^n), \delta G) \to \Set^{G \times \Delta^{op}}(\delta F \times \Delta^n, \delta G)$
\end{center}

is a bijection.
\end{lem}

\begin{proof}
We need to show that every map $h: \delta F \times \Delta^n \to \delta G$ factors uniquely through $\theta$, that is for every $m \in \Delta^{op}$ there exists a unique dotted arrow such that the diagram

\begin{center}
\begin{tikzcd}
\coprod_{m \to n} \gamma F_m \ar[d, "\theta", swap] \ar[r, "h"] & \gamma G_m\\
\gamma (\coprod_{m \to n} F_m) \ar[ur, dotted]
\end{tikzcd}
\end{center}

commutes. For each $\alpha: m \to n$, there are coproduct inclusions $i_{\alpha}: F_m \to \coprod_{m \to n} F_m$ and $j_{\alpha}: \gamma F_m \to \coprod_{m \to n} \gamma F_m$. Moreover, $\theta$ is defined to be the unique map such that $\theta \circ j_{\alpha} = \gamma i_{\alpha}$ for every $\alpha$.\\
Since $\gamma$ is full, there is $t_{\alpha}: F_m \to G_m$ such that $\gamma t_{\alpha} = h \circ j_{\alpha}$. By universal property of coproducts, there is a unique $s: \coprod_{m \to n} F_m \to G_m$ such that the diagram

\begin{center}
\begin{tikzcd}
F_m \ar[d, "i_{\alpha}"] \ar[r, "t_{\alpha}"] & G_m\\
\coprod_{m \to n} F_m \ar[ur, dotted, "s", swap]
\end{tikzcd}
\end{center}

commutes. Now, for every $\alpha$ we have

\begin{center}
$\gamma s \circ \theta \circ j_{\alpha} = \gamma s \circ \gamma i_{\alpha} = \gamma (s \circ i_{\alpha}) = \gamma t_{\alpha} = h \circ j_{\alpha}$
\end{center}

which, by universal property of coproducts, implies $\gamma s \circ \theta = h$. This proves that $\theta^{\ast}$ is surjective.\\
To prove that it is injective, we will show that if there is another map $r: \gamma (\coprod_{m \to n} F_m) \to \gamma G_m$ such that $r \circ \theta = h$ then it must be $r = \gamma s$. Since $\gamma$ is full, there is $s': \coprod_{m \to n} F_m \to G_m$ such that $\gamma s' = r$. Now, for each $\alpha$ compute

\begin{center}
$\gamma(s' \circ i_{\alpha}) = \gamma s' \circ \gamma i_{\alpha} = r \circ \theta \circ j_{\alpha} = h \circ j_{\alpha} = \gamma t_{\alpha}$
\end{center}

and, since $\gamma$ is faithful, this implies that $s' \circ i_{\alpha} = t_{\alpha}$. But $s$ was defined as the unique map such that $s \circ i_{\alpha} = t_{\alpha}$ for every $\alpha$, so we have $s = s'$ and the proof is complete.
\end{proof}

\begin{prop} \label{deltaenrichedfff}
The functor $\delta$ can be endowed with the structure of a simplicial functor, which also makes it fully faithful in the enriched sense.
\end{prop}

\begin{proof}
Choose two functors $F$ and $G$ and $n \in \Delta^{op}$. The simplicial structure of $\delta$ at the $n$-th dimension is constructed as follows:

\begin{center}
\begin{tikzcd}
\Set^{\mathcal{C} \times \Delta^{op}}(F \times \Delta^n, G) \ar[r, "\delta"] & \Set^{G \times \Delta^{op}}(\delta(F \times \Delta^n), \delta G) \ar[r, "\theta^{\ast}"] & \Set^{G \times \Delta^{op}}(\delta F \times \Delta^n, \delta G)
\end{tikzcd}
\end{center}

where of course the first and the third term are respectively the $n$-th dimensions of the simplicial enrichments.\\
It remains to show that this composite is a bijection, but the arrow labeled $\delta$ is so by \Cref{deltaunenrichedfff}, and $\theta^{\ast}$ is bijective by \Cref{keylemma}.
\end{proof}

\begin{lem} \label{altinftylocalization}
Let $P$ be a fibrant replacement in the model structure on simplicial categories. Given a simplicial model category $\mathbf{N}$, there is an equivalence of $\infty$-categories

\begin{center}
$\N_{\Delta}(\mathbf{N}^{\circ}) \simeq \N_{\Delta}(P \mathbf{N})$.
\end{center}
\end{lem}

\begin{proof}
Consider the composite

\begin{center}
$\mathbf{N}^{\circ} \hookrightarrow \mathbf{N} \to P \mathbf{N}$
\end{center}

where the first map is the inclusion and the second is a trivial cofibration. The inclusion is in particular a weak equivalence, because every object of $\mathbf{N}$ is weakly equivalent to a bifibrant one, so the composite is a weak equivalence by the 2-out-of-3 property. Now observe that $\N_{\Delta}$ is a right Quillen functor, so by Ken Brown's lemma it preserves all weak equivalences between fibrant objects. Since both $\mathbf{N}^{\circ}$ and $P \mathbf{N}$ are fibrant, applying $\N_{\Delta}$ to the composite map above gives the desired equivalence.
\end{proof}

\begin{proof}[Proof of \Cref{i2.65}]
By \Cref{accessibletopresheaves}, we restrict to the case where $\mathcal{K}$ is presentable. In this case, \cite{HTT}, Theorem A.3.7.6 says that we can find a combinatorial simplicial model category $\mathbf{M}$ such that $\mathcal{K} \simeq \N_{\Delta}(\mathbf{M}^{\circ})$. Using \Cref{combinatorialtopresheaves}, we can further find a functor of the form $\mathbf{M} \to \sSet^{\mathcal{C}}$, where $\mathcal{C}$ is an ordinary small category, such that the induced $\N_{\Delta}(\mathbf{M}^{\circ}) \to \N_{\Delta}((\sSet^{\mathcal{C}})^{\circ})$ is fully faithful.\\
We are now reduced to showing that $\N_{\Delta}((\sSet^{\mathcal{C}})^{\circ})$ admits a full embedding as required. Now, the inclusion $(\sSet^{\mathcal{C}})^{\circ} \hookrightarrow \sSet^{\mathcal{C}}$ is a weak equivalence of simplicial categories (because every object is weakly equivalent to a bifibrant one), and moreover by \Cref{deltaenrichedfff} there is a functor $\sSet^{\mathcal{C}} \hookrightarrow \sSet^G$ which is a full embedding in the simplicially enriched sense. Finally, there is a trivial cofibration of simplicial categories $\sSet^G \to P \sSet^G$, where the codomain is a fibrant replacement in the model structure on simplicial categories.\\
We now have a simplicially enriched full embedding $(\sSet^{\mathcal{C}})^{\circ} \hookrightarrow P \sSet^G$, and since the mapping spaces of simplicial nerves are computed as mapping spaces of the original simplicial categories, this induces a full embedding $\N_{\Delta}((\sSet^{\mathcal{C}})^{\circ}) \hookrightarrow \N_{\Delta}(P \sSet^G)$.\\
We conclude by observing that, by \Cref{altinftylocalization}, the codomain of this functor is equivalent to $\N_{\Delta}((\sSet^G)^{\circ})$ which, in turn is equivalent to $\mathcal{S}^G$ by \cite{HTT}, Proposition 4.2.4.4.
\end{proof}


\end{document}